\documentclass[11pt,reqno]{amsart}

\usepackage[colorlinks]{hyperref}

\usepackage{multirow}

\usepackage{graphicx}

\usepackage{style}

\usepackage{textcomp}

\usepackage{float}

\usepackage[final]{showkeys}

\usepackage[table]{xcolor}

\newcommand{\vertiii}[1]{{\left\vert\kern-0.25ex\left\vert\kern-0.25ex\left\vert #1
    \right\vert\kern-0.25ex\right\vert\kern-0.25ex\right\vert}}

\usepackage{tikz}
\usetikzlibrary{shapes.geometric, arrows, positioning}

\tikzset{
startstop/.style={
    rectangle,
    rounded corners,
    minimum width=3cm, 
    minimum height=1cm,
    align=center, 
    draw=black, 
    fill=blue!30
    },
io/.style={
    trapezium, 
    trapezium left angle=70, 
    trapezium right angle=110, 
    minimum width=1cm, 
    minimum height=1cm, 
    text centered, 
    draw=black, 
    fill=red!30
    },
process/.style={
    rectangle, 
    minimum width=3cm, 
    minimum height=1cm, 
    text centered, 
    draw=black, 
    fill=orange!30
    },
decision/.style={
    diamond, 
    minimum width=3cm, 
    minimum height=1cm, 
    text centered, 
    draw=black, 
    fill=green!30
    },
arrow/.style={
    draw,
    thick,
    ->,
    >=stealth
    },
}

\usepackage{enumerate}

\usepackage{algorithmic}

\usepackage{algorithm}

\usepackage{resizegather}

\usepackage{perpage}
\MakePerPage{footnote}

\usepackage{caption}
\usepackage[labelfont=rm]{subcaption}

\hypersetup{pdffitwindow=true,linkcolor=blue,citecolor=blue,urlcolor=blue,menucolor=blue}

\usepackage{comment}

\usepackage{arydshln}

\begin{document}

\title[A Taxonomy of Crystallographic Sphere Packings]{A Taxonomy of \\Crystallographic Sphere Packings}

\author{Debra Chait}
\email{debra.chait@macaulay.cuny.edu}

\author{Alisa Cui}
\email{alisa.cui@yale.edu}

\author{Zachary Stier}
\email{zstier@princeton.edu}
\thanks{This work was supported by NSF grant DMS-1802119 at the DIMACS REU hosted at Rutgers University--New Brunswick.}

\date{July 27, 2018.}


\begin{abstract}
The Apollonian circle packing, generated from three mutually-tangent circles in the plane, has inspired over the past half-century the study of other classes of space-filling packings, both in two and in higher dimensions. Recently, Kontorovich and Nakamura introduced the notion of crystallographic sphere packings, $n$-dimensional packings of spheres with symmetry groups that are isometries of $\H^{n+1}$.
There exist at least three sources which give rise to crystallographic packings, namely polyhedra, reflective extended Bianchi groups, and various higher dimensional quadratic forms.
When applied in conjunction with the Koebe-Andreev-Thurston Theorem, Kontorovich and Nakamura's Structure Theorem guarantees crystallographic packings to be generated from polyhedra in $n=2$. The Structure Theorem similarly allows us to generate packings from the reflective extended Bianchi groups in $n=2$ by applying Vinberg's algorithm to obtain the appropriate Coxeter diagrams. In $n>2$,
the Structure Theorem when used with Vinberg's algorithm allows us to explore whether certain Coxeter diagrams in $\mathbb{H}^{n+1}$ for a given quadratic form admit a packing at all. 
Kontorovich and Nakamura's Finiteness Theorem shows that there exist only finitely many classes of superintegral such packings, all of which exist in dimensions $n\le20$. 
In this work, we systematically determine all known examples of crystallographic sphere packings.

\end{abstract}

\maketitle

\tableofcontents

\section{Introduction}
\label{intro}

\begin{figure}[h]
\resizebox{!}{.5\textheight}{
\begin{tikzpicture}[node distance=2cm]

\node (Poly) [startstop] {Polyhedra (\secref{polyhedral packings})};
\node (Bianchi) [startstop, right of=Poly, xshift=4.5cm] {Dimension $n\ge3$\footnotemark\ (\secref{hd})};
\node (High) [startstop, right of=Bianchi, xshift=3cm] {Bianchi groups (\secref{bianchi packings})};

\node (KAT) [io, below of=Poly] {Apply K-A-T Theorem};
\node (quad) [process, below of=Bianchi] {Select quadratic form};
\node (Vin) [process, below of=quad] {Apply Vinberg's algorithm \cite{Vin72}};
\node (fund) [process, below of=Vin] {Obtain fundamental polyhedron};
\node (Cox) [process, below of=fund] {Describe with Coxeter diagram};
\node (Struc) [io, below of=Cox] {Apply Structure Theorem \cite{KN18}};

\node (pack) [startstop, below of=Struc] {Generate circle packing};

\draw [arrow] (Bianchi) -- (quad);
\draw [arrow] (High) |- (quad);
\draw [arrow] (quad) -- (Vin);
\draw [arrow] (Vin) -- (fund);
\draw [arrow] (fund) -- (Cox);
\draw [arrow] (Cox) -- (Struc);
\draw [arrow] (Struc) -- (pack);
\draw[arrow] (Poly) -- (KAT);
\draw[arrow] (KAT) |- (Struc);

\draw[thick,blue,arrow] (High) |- node[anchor=west,yshift=3cm]{\cite{BM13,Mcl10,Vin72}}(Cox);

\end{tikzpicture}
}
\caption{An outline of how each of our packings arises.}
\label{flow}
\end{figure}
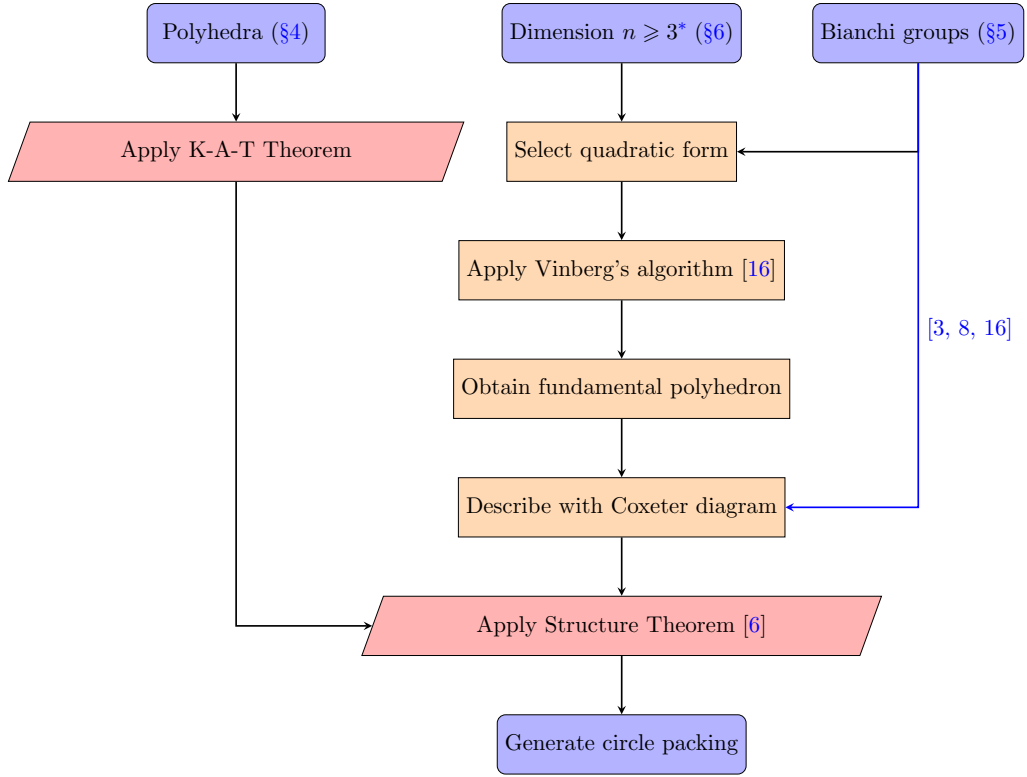

\smallskip \footnotetext{The arrow connecting \emph{Bianchi groups} to \emph{Coxeter diagram} via \cite{BM13,Mcl10,Vin72} should be taken to connect \emph{Dimension $n\ge3$} to \emph{Coxeter diagram} as well. This arrow indicates that our research relied on Belolipetsky \& McLeod's and Vinberg's conversions of Bianchi groups and higher dimensional forms into Coxeter diagrams, performed through the steps indicated.}

\begin{Def}[sphere packing]\label{packing}
A \emph{sphere packing} in $\R^n \cup \{\infty\}$ is a collection of spheres that: 
\begin{itemize}
	\item are oriented to have mutually disjoint interiors, and
	\item densely fill up space, so that any ball in $\R^n$ intersects the interior of some sphere in the packing. 
\end{itemize} 
\end{Def}

\begin{Def}[crystallographic sphere packing]\label{cryst packing}
A \emph{crystallographic sphere packing} in $\R^n$ is a sphere packing generated by a finitely generated reflection group $\Gamma < \isom(\H^{n+1})$ \cite{KN18}.
\end{Def}

The Structure Theorem from \cite{KN18} allows us to identify crystallographic sphere packings as finite collections of generating spheres. 

\begin{thm}[Structure Theorem for Crystallographic Packings]\label{structthm}
Consider a finite collection of $n$-spheres $\widetilde\cC$, called the \emph{supercluster}, where $\widetilde\cC$ can be decomposed into finite collections $\cC,\hat\cC$, called the \emph{cluster} and \emph{cocluster}, respectively, with $\widetilde\cC=\cC \sqcup \hat\cC$. Suppose $\Gamma=\left\langle\hat\cC\right\rangle$ acts on $\H^{n+1}$ with finite covolume. If $\widetilde\cC$ satisfies:
\begin{itemize}
    \item any two spheres in $\cC$ are disjoint or tangent, and 
    \item every sphere in $\cC$ is disjoint, tangent, or orthogonal to any sphere in $\hat\cC$,
\end{itemize}
then $\Gamma$ produces a crystallographic sphere packing via $\Gamma\cdot\cC$. Conversely, every crystallographic packing arises in this way. 
\end{thm}

\begin{Def}[superpacking]\label{super}
A \emph{superpacking} is a configuration of spheres generated by the action  $\widetilde\Gamma \cdot \cC$, where $\widetilde\Gamma=\left\langle\hat\cC,\cC\right\rangle$. 
\end{Def}

Note that this is not a packing in the sense of \defref{packing} because the interiors are not necessarily mutually disjoint \cite{KN18}.

To study sphere packings, we identify the \textit{bend} of a sphere as the inverse of the radius.
\begin{Def}[integral, superintegral]
If every sphere in a crystallographic packing has integer bend, then it is an \emph{integral} packing. If every sphere in a superpacking has integer bend, then it is a \emph{superintegral} packing. 
\end{Def}

\smallskip The following theorem from \cite{KN18} motivates our work towards classifying all crystallographic sphere packings.
\begin{thm}[Finiteness Theorem] \label{finite}
Up to commensurability of $\widetilde\Gamma$, there are finitely many superintegral crystallographic sphere packings, all of which exist in dimension $n < 21$. 
\end{thm}

In this paper, we categorize the integrality and nonintegrality of Bianchi packings in \secref{bianchi packings} and \secref{int and nonint bianchi}, and list all known examples and identify a new integral polyhedron in \secref{int poly}, a packing in dimension two in \secref{beyond} and packings in dimensions four, five, six, seven, nine, ten and 12 in \secref{hdpackingdata}.

\section{Further Objects}\label{fo}

\begin{Def}[oriented spheres]\label{orientedspheres}
For $r\in\R\backslash\{0\}$, the sphere centered at $z$ with radius $r$ is the set $\partial B_z(r)$, and we define its interior to be $\{x\in\hat{R^n}\mid (r-\abs{z-x})\sign{r}>0\}$. 
\end{Def}

\begin{Def}[circle inversion/reflection]\label{inversion}
To invert about $\partial B_z(r)$ (the points in $\hat{\R^n}$ at distance exactly $r$ from $z$), send the point $x\in\hat{\R^n}$ at distance $d=\abs{x-z}$ from $z$ to the point on the ray through $x$ beginning at $z$ at distance $\frac{r^2}{d}$. (This also swaps $z$ and $\infty$.)
\end{Def}

A symmetric $(n+2) \times (n+2)$ matrix $Q$ with signature $(1,n+1)$ gives rise to a model of hyperbolic space through one sheet of the two sheeted hyperboloid $\{x\in\R^{n+2}\mid \ip{x}{x}_Q=-1\}$, where $\ip{x}{y}_Q=xQy^T$. We use here \begin{align}\label{Q}
    Q=Q_n=\begin{pmatrix}&\half\\\half\\&&-I_n\end{pmatrix},
\end{align} where the subscript $n$ may be omitted depending on context. 

\begin{lemma}
When viewed in an upper half-space model, $\H^{n+1}$'s planes are precisely the hemihyperspheres with circumferences on the boundary of space, namely $\R^n$. In the case of a hyperplane as the boundary, the plane is a plane in the Euclidean sense which is orthogonal to the boundary hyperplane. 
\end{lemma}

\begin{Def}[inversive coordinates]\label{invcoords}
An oriented sphere centered at $z$ with radius $r\in\R\backslash\{0\}$ may be represented by \emph{inversive coordinates} consisting of $\left(\hat{b},b,bz\right)$ for
\begin{align*}
    b=\frac{1}{r}\text{ and }\hat{b}=\frac{1}{\hat{r}},
\end{align*}
where $\hat{r}$ is the oriented radius of $\partial B_z(r)$ reflected through $\partial B_0(1)$. We refer to $b$ as the \emph{bend} and $\hat{b}$ as the \emph{co-bend}.
\end{Def}

As shown in \cite{Kon17}, any $n$-dimensional inversive coordinate $v$ satisfies $\ip{v}{v}_Q=-1$. This leads to the following definition for reflecting about an oriented sphere: 

\begin{Def}[reflection matrix]\label{reflmatrix}
The \emph{reflection matrix} about $\hat{v}$ is given by $R_{\hat{v}}=I_n+2Q\hat{v}^T\hat{v}$. 
\end{Def}

This arises from the formula for reflection of $v$ about $\hat{v}$ with the inner product $\ip{\cdot}{\cdot}$ given as $R_{\hat{v}}(v)=v-2\frac{\ip{v}{\hat{v}}}{\ip{\hat{v}}{\hat{v}}}\hat{v}$ which expands in our inner product as $R_{\hat{v}}(v)=v+2vQ\hat{v}^T\hat{v}$ and is a right-acting matrix on $v$. 

We are also equipped to use inversive coordinates to represent ``degenerate spheres" of ``radius infinity," i.e. codimension-1 hyperplanes in $\R^n$. 

\begin{lemma}\label{planefacts}
Consider a hyperplane $H$ with $\codim H=1$, normal vector $\hat{n}$, and $P\in H$ the closest point to the origin, and let $S_r\subset\R^n$ be the sphere of radius $r$ tangent to $P$ with interior on opposite half-planes from the origin. Then, 
\begin{align}
\label{hyperplane1}\lim\limits_{r\to\infty}bz&=\hat{n},\\
\label{hyperplane2}\lim\limits_{r\to\infty}\hat{b}&=2\abs{P}
\end{align}
for $z,b,\hat{b}$ dependent on $r$. 
\end{lemma}

This enables us to legitimately view hyperplanes as the limits of increasingly large spheres. 

\begin{Def}[Coxeter diagram]\label{Cox}
A \emph{Coxeter diagram} encodes the walls of a Coxeter polyhedron (whose dihedral angles are all of the form $\frac{\pi}{n}$) as nodes in a graph, where we draw between nodes corresponding to walls meeting at dihedral angle $\theta$ if they meet at all   
\[
\begin{cases}
\text{a thick line}, &  \text{if walls are tangent at a point (including $\infty$).} \\
\text{no line}, & \text{if } \theta=\frac{\pi}{2}. \\
\text{a dashed line}, & \text{if walls are disjoint}. \\
n-2 \text{ lines}, & \text{if } \theta=\frac{\pi}{n}. \\
\end{cases}
\]
\end{Def}

By \thmref{structthm}, a Coxeter diagram can be used to visually identify clusters by identifying those vertices that are adjacent to all other vertices exclusively by thick, dashed, or no lines.

\begin{Def}[Gram matrix]\label{Gram}
If $V$ is a rank-$(n+2)$ matrix of inversive coordinates, then its \emph{Gram matrix} is defined as $VQV^T$. 
\end{Def}

The rows and columns of a Gram matrix correspond to walls of a Coxeter polyhedron, where the entries are determined by 
\[
G_{i,j}=\ip{v_i}{v_j}_Q
=
\begin{cases}
-1, &  v_i=v_j. \\
1, & v_i || v_j. \\
0, & v_i \perp v_j. \\
\cos(\theta), & \theta_{v_i,v_j}. \\
\cosh(d), & d=\text{hyperbolic distance}(v_i,v_j).
\end{cases}
\]

A Gram matrix encodes the same information as a Coxeter diagram, but also includes the hyperbolic distance between two disjoint walls. 

\begin{Def}[bend matrix]\label{bend matrix}
For $V$, a rank-$(n+2)$ collection of inversive coordiantes, and $R$, the reflection matrix about a $n$-sphere, a left-acting \emph{bend matrix} $B$ satisfies the equation 
\begin{align}\label{BV}
    BV=VR.
\end{align} 
\end{Def}

Bend matrices can be used to compute the inversive coordinates of a packing. They are a useful tool in proving integrality of packings, as will be demonstrated in \secref{gm}.

We use three sources to generate crystallographic packings, whose details will be elaborated upon in the coming sections. \figref{flow} provides a rough outline for how these packings can be obtained.

\section{General Methods}\label{gm}

\subsection{Producing crystallographic packings} \label{producing cryst packing}
Every cluster $\cC$ identified from one of our three sources above was used to produce a crystallographic packing by applying \thmref{structthm}. 
To do so, all circles in the identified $\cC$ were reflected about circles in $\hat{\cC}$. For each $\cC$, we built an inversive coordinate matrix $V$, which we reflected about all 
$\hat{v} \in \hat{\cC}$ by 
\begin{equation}\label{right acting refl}
  R_{\hat{v}}(V)=VR  
\end{equation}
to obtain the inversive coordinates of the next generation of circles in the packing. To obtain further generations, each new circle produced by (\ref{right acting refl}) was reflected about all $\hat{v} \in \hat{\cC}$ in a similar manner, the infinite repetition of which produces a crystallographic packing.

\smallskip Diagrams of the packings were produced by applying Mathematica's graphics features to the list of inversive coordinates of the packing. 

\subsection{Proving integrality, nonintegrality, and superintegrality} \label{proving int, nonint, supint}
One feature of crystallographic packings to study is the bend of each sphere in the packing.

Finding every integral (and superintegral) crystallographic packing is of fundamental interest, and a main objective of our study. The following lemmata outline our general methods of proving integrality, non-integrality, and superintegrality of crystallographic packings.

\begin{lemma} \label{rescaling}
There is always a transformation which scales the bends of all circles in a packing by some constant.
\end{lemma}

\begin{proof}
Let $d_1=|z|-r$ be the point on a circle $s$ closest to the origin, and let $d_2=|z|+r$ be the point on $s$ furthest from the origin. Inversion through the unit circle sends $d_1 \mapsto \frac{1}{d_1}$ and $d_2 \mapsto \frac{1}{d_2}$. Subsequent inversion through a circle of radius $\alpha$ centered at the origin sends $\frac{1}{d_1} \mapsto \alpha^2 d_1 = \alpha^2(|z|-r)$ and $\frac{1}{d_2} \mapsto \alpha^2{d_2} = \alpha^2(|z|+r)$. Thus, the new circle has radius $\alpha^2 r$ and any circle can be rescaled by choice of $\alpha$.
\end{proof}

As a consequence, if a packing has ``bounded rational" bends---i.e., no bend in the packing has denominator greater than some upper bound---then there is a conformally equivalent integral packing.

\begin{lemma} \label{bendint}
If all bend matrices of a cluster $\cC$ are integral and the bend of each circle in $\cC$ is rational, then the packing generated by $\hat{\cC}$ on $\cC$ is integral. 
\end{lemma}

\begin{proof}
If all bends in $\cC$ are integral, then the action $BV=VR$ is always an integral linear combination of integers, and therefore integral. Otherwise, \lemref{rescaling} allows a rescaling of the bends to integers.
\end{proof}
Note that \lemref{bendint} can hold even if the bend of each circle in $\cC$ is irrational in the case that the bends can be uniformly rescaled by \lemref{rescaling} to achieve integrality.

\begin{lemma} \label{nonintpoly}

Let V be an $m \times (n+2)$ matrix of inversive coordinates corresponding to a cluster $\cC$ of $m$ circles. If there exists a square matrix $g$ satisfying $gV=0$ 
with a nonrational (implying also nonintegral) linear relationship between some two entries in any row, then $\cC$ cannot be integral.
\end{lemma}

\begin{proof}
A nonintegral relation between the entries of the bend matrices precludes the possibility of an integral packing, since the packing is entirely generated by reflections, namely, multiplication with its bend matrices.
\end{proof}

The following theorem from \cite{KN18} relates superintegrality to arithmeticity as defined by Vinberg's arithmeticity criterion \cite{Vin67}.
\begin{thm} \label{superint}
If a packing is superintegral, then the group $\widetilde\G$ generated by reflections through $\cC$ and $\hat\cC$ is arithmetic.
\end{thm}

\section{Polyhedral Packings}\label{polyhedral packings}
A version of the Koebe-Andreev-Thurston theorem allows polyhedra (equivalently, 3-connected planar graphs) to be realized as circle packings.
\begin{thm}[Koebe-Andreev-Thurston Circle Packing Theorem]
Every 3-connected planar graph can be realized as a polyhedron with a midsphere, and this realization is unique up to conformal equivalence.
\end{thm}
Here, a \textit{midsphere} is a sphere tangent to every edge of a polyhedron $\Pi$. Its dual polyhedron $\hat\Pi$ has the same midsphere. A realization of $\Pi$, $\hat\Pi$, and their midsphere gives rise to two clusters of circles (see \figref{midsphere}) which pass through edge tangency points and have normal vectors along the rays connecting vertices (of both $\Pi$ and $\hat\Pi$) to the center of the midsphere.

Stereographic projection of these circles onto $\R^2 \cup \{\infty\}$ yields a collection of circles which by \thmref{structthm} can be viewed as a cluster-cocluster pair $\cC$, $\hat\cC$ giving rise to a circle packing: call circles in $\cC$ those centered around vertices of $\Pi$ and circles in $\hat\cC$ those centered around vertices of $\hat\Pi$. Any two circles in $\cC$ are either tangent or disjoint, and every circle in $\hat\cC$ is only tangent, disjoint, or orthogonal to circles in $\cC$. The packing produced by a polyhedron $\Pi$ is called $\mathscr{P}$; similarly a superpacking is called $\nye{\mathscr{P}}$.

\begin{figure}[p] 
\centering
\begin{minipage}[b]{0.49\linewidth}
 \includegraphics[width=1\textwidth]{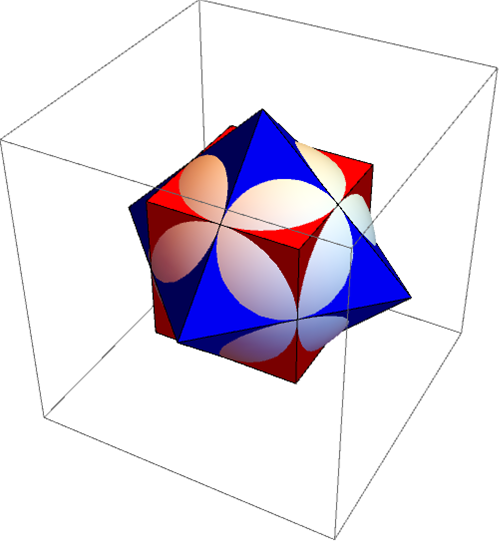}
 \caption{Octahedron (blue), its dual (red cube), and midsphere (grey)}
 \label{midsphere}
\end{minipage}
\begin{minipage}[b]{0.49\linewidth}
 \includegraphics[width=1\textwidth]{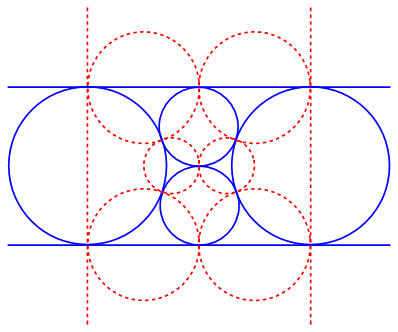}
 \caption{Stereographic projection onto $\R^2$}
\end{minipage}
\includegraphics[width=1\textwidth]{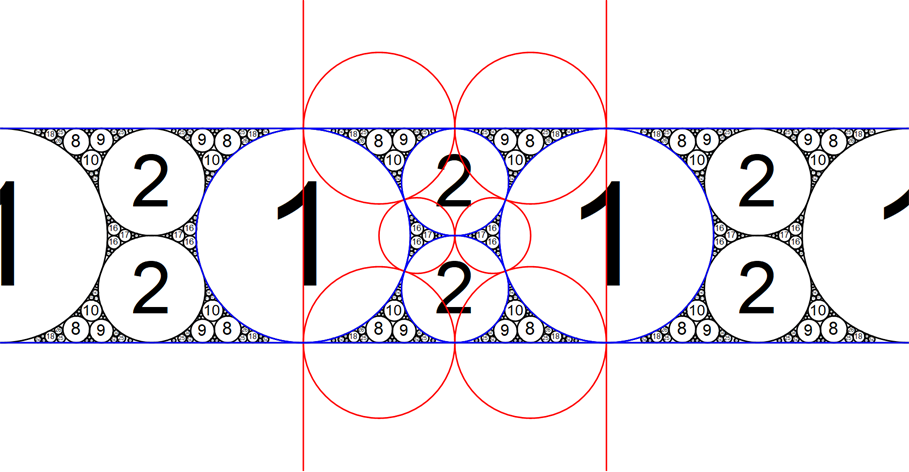}
\caption{Packing with bends.}

\end{figure}

Previous work has classified certain types of polyhedra, for example \textit{uniform polyhedra}: those whose faces are regular polygons and which are vertex-transitive. From Kontorovich and Nakamura we know the following theorem.

\begin{thm} \label{uniform}
    The only integral uniform polyhedra are:
    \begin{itemize}
        \item (Platonic) tetrahedron, octahedron, cube;
        \item (Archimedean) cuboctahedron, truncated tetrahedron, truncated octahedron;
        \item (prisms/antiprisms) 3,4,6-prisms, 3-antiprism.
    \end{itemize}
\end{thm}

\subsection{Methods}
We were able to systematically generate polyhedron raw data using the program \texttt{plantri}. Mathematica programs turned data into packings using techniques outlined in \cite{BS04} (see also \cite{Zie04,Riv86,Riv94,Ver91}). Currently all polyhedra are documented on vertices $n \leq 7$, with some additional larger regular polyhedra. 

We identify 2 broad categories of polyhedra which branch into 4 total smaller subcategories: integral-superintegral, integral but not superintegral, nonintegral-rational, and nonintegral-nonrational. To more accurately define the relationships between polyhedra, we introduce a \textit{gluing operation}.

\begin{Def}[gluing operation]\label{gluingop}
Polyhedra can be glued along faces or vertices. Let $A$ be a polyhedron with vertex set $V_A$, edges $E_A$, and faces $F_A$. Similarly let $B=\{V_B,E_B,F_B\}$. A face-face gluing operation is only valid if two $n$-gon faces are equivalent: the same types of faces, in the same order, are adjacent to both. A vertex-vertex gluing operation is only valid if two vertices of degree $n$ are equivalent: they lie on the same type of faces, in the same order. 
\begin{itemize}
    \item To glue faces $f_a \in A$ and $f_b \in B$ : let $f_a, f_b$ be $n$-gons bounded by vertices $\{v_{a_1}, \dots ,v_{a_n}\}, \{v_{b_1}, \dots ,v_{b_n}\}$ and edges $\{e_{a_1}, \dots ,e_{a_n}\} , \{e_{b_1}, \dots ,e_{b_n}\}$. Vertices and edges must be glued together in a one-to-one mapping with stretching distortions only in the plane of faces $f_a,f_b$, which are ommitted from the final polyhedron.
    \item To glue two vertices of equal degree $n$: let $v_a \in A$ and $v_b \in B$ have edges $\{e_{a_1}, \dots ,e_{a_n}\}$ and $\{e_{b_1}, \dots ,e_{b_n}\}$. Edges are joined in a one-to-one mapping creating new faces bounded by preexisting edges from $A$ and $B$ such that new face $m$ is bounded by the union of all edges on faces $f_{a_m} \in A$ and $f_{b_m} \in B$, and dropping both $v_a$ and $v_b$ in the final polyhedron.
\end{itemize}
\end{Def}

\begin{figure*}[ht]
\centering
\begin{minipage}[b]{0.24\linewidth}
 \includegraphics[width=1\textwidth]{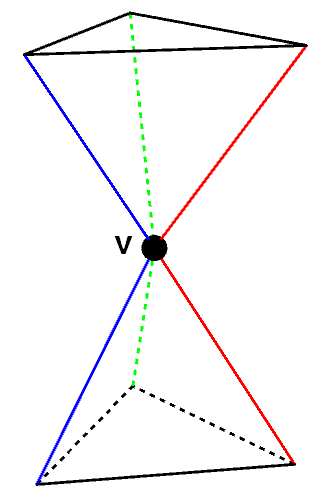}
 \caption*{(a)}
\end{minipage}
\begin{minipage}[b]{0.24\linewidth}
 \includegraphics[width=1\textwidth]{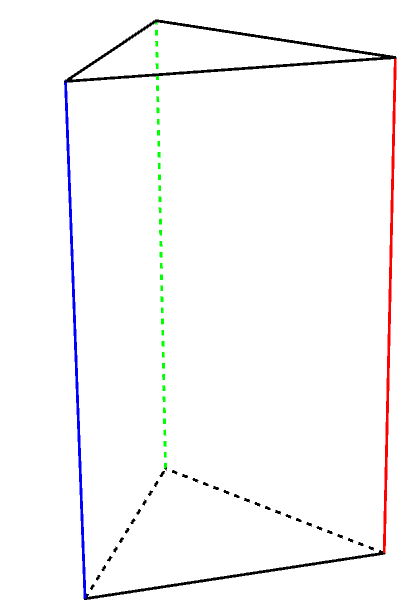}
  \caption*{(b)}
\end{minipage}
\begin{minipage}[b]{0.24\linewidth}
 \includegraphics[width=1\textwidth]{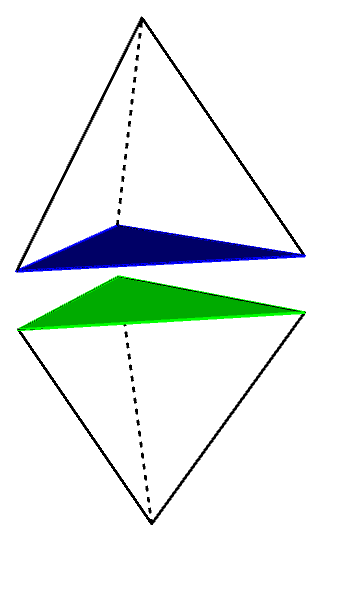}
  \caption*{(c)}
\end{minipage}
\begin{minipage}[b]{0.24\linewidth}
 \includegraphics[width=1\textwidth]{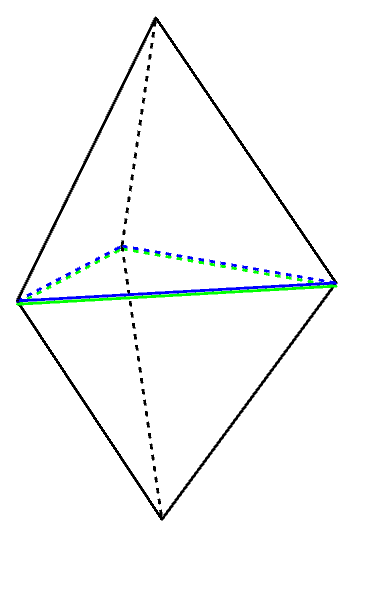}
  \caption*{(d)}
\end{minipage}
\caption{Gluing two tetrahedra at vertex {\fontfamily{phv}\selectfont v} (a) to produce a triangular prism (b), gluing two tetrahedra along the blue and green faces (c) to produce triangular bipyramid (d)}
\end{figure*}

\begin{figure*}[h!]
\centering
\begin{minipage}[b]{0.49\linewidth}
 \includegraphics[width=1\textwidth]{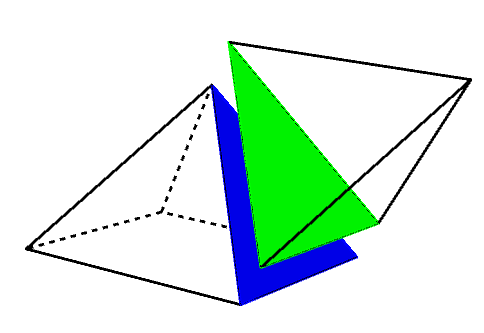}
 \caption*{(a)}
\end{minipage}
\begin{minipage}[b]{0.49\linewidth}
 \includegraphics[width=1\textwidth]{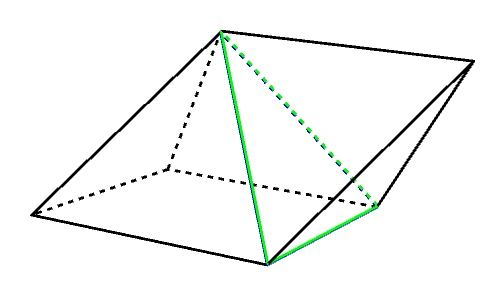}
  \caption*{(b)}
\end{minipage}
\begin{minipage}[b]{1\linewidth}
 \includegraphics[width=1\textwidth]{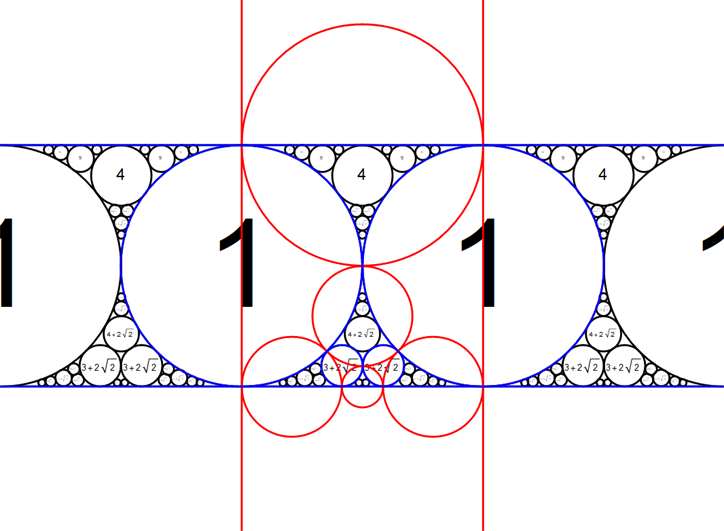}
  \caption*{(c)}
\end{minipage}
\caption{(a) Gluing a triangular face of a square pyramid (blue) to a tetrahedron (green) produces (b) whose packing, (c), is not integral as this is not a valid gluing.}
\end{figure*}

\subsection{Integral Polyhedra}\label{int poly}
A polyhedron is called \textit{integral} if it has some associated integral packing.

\begin{thm} \label{intpoly}
There are exactly 4 unique, nondecomposable (not the result of some series of gluing operations) integral polyhedra with $n \leq 7$ vertices: tetrahedron, square pyramid, hexagonal pyramid, and unnamed 6v7f\_2. We call them \textit{seed polyhedra}, as in \cite{KN18}.
\end{thm} 

\begin{figure}[H]
\centering
\begin{minipage}[b]{0.25\linewidth}
 \includegraphics[width=1\textwidth]{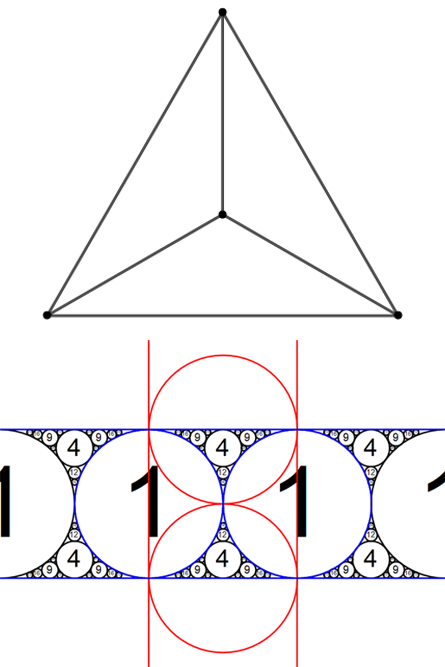}
 \caption*{(a)}
\end{minipage}
\begin{minipage}[b]{0.22\linewidth}
 \includegraphics[width=1\textwidth]{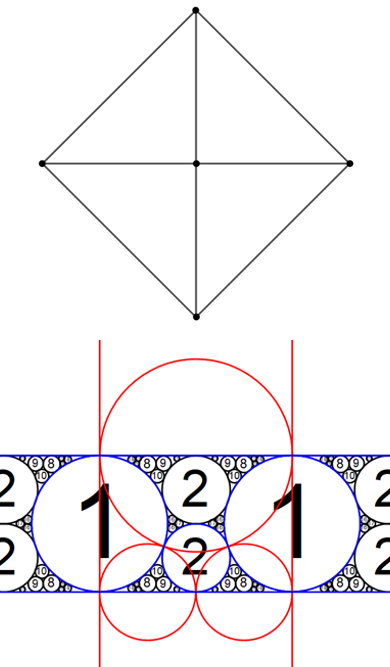}
  \caption*{(b)}
\end{minipage}
\begin{minipage}[b]{0.26\linewidth}
 \includegraphics[width=1\textwidth]{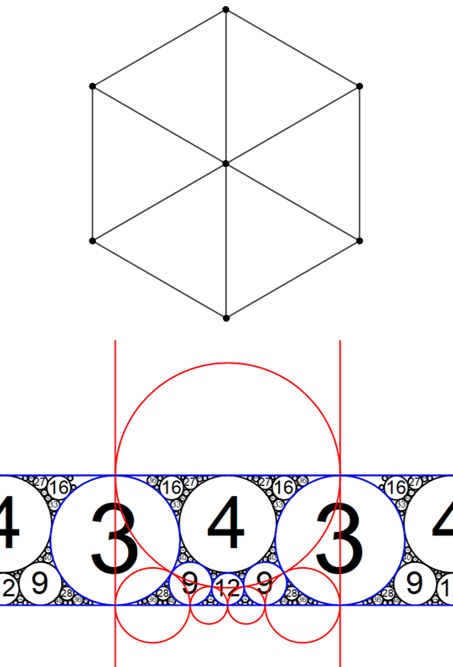}
  \caption*{(c)}
\end{minipage}
\begin{minipage}[b]{0.22\linewidth}
 \includegraphics[width=1\textwidth]{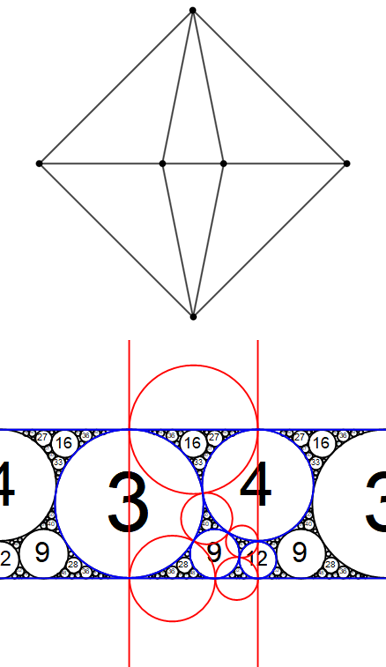}
  \caption*{(d)}
\end{minipage}
\caption{(a) Tetrahedron, (b) square pyramid, (c) hexagonal pyramid, and (d) 6v7f\_2}
\end{figure}

This proof of this theorem relies on the following lemma.
\begin{lemma} \label{gluingaddition}
A gluing operation of $A$ onto $B$ yields a polyhedron with strictly more vertices, edges, and faces than either $A$ or $B$. In particular,
\begin{itemize}
    \item gluing $A$ and $B$ along an $n$-gon face yields polyhedron $C$ such that $|V_C|=|V_A| + |V_B| - n$, $|E_C| = |E_A| + |E_B| - n$, $|F_C| = |F_A| + |F_B| - 2$, and
    \item gluing $A$ and $B$ at a vertex of degree $n$ yields polyhedron $C$ such that $|V_C| = |V_A| + |V_B| - 2$, $|E_C| = |E_A| + |E_B| - n$, $|F_C| = |F_A| + |F_B| - n$.
\end{itemize}
\end{lemma}

\begin{proof}[Proof of \thmref{intpoly}]
Aside from the polyhedra specifically mentioned in \thmref{intpoly}, only 9 of size $n \leq 7$ vertices are integral. Methods described in \lemref{nonintpoly} and \lemref{qnonintpoly} are used to show that all others are not integral; the series of gluings used to construct the 9 others are detailed in \secref{integralpoly}. What remains is to show that the tetrahedron, square pyramid, hexagonal pyramid, and 6v7f\_2 cannot be constructed from a series of gluing operations.
\begin{itemize}
    \item A tetrahedron is the smallest possible polyhedron; by \lemref{gluingaddition} it cannot be the result of gluings.
    \item The square pyramid is only larger than a tetrahedron, but gluing two tetrahedra yields (along a face) $|V| = 5$,  $|E| = 9$,  $|F| = 6$ or (along a vertex) $|V| = 6$, $|E| = 9$, $|F| = 5$; a square pyramid has 8 edges.
    \item 6v7f\_2 does not arise from gluing two tetrahedra (above). Gluing two square pyramids yields either $|V| = 8$ or $|F| = 8$; gluing a tetrahedron to a square pyramid (by symmetry) has only two possibilities, one with $|V| = 7$, one shown in [insert figure]. 
    \item Hexagonal pyramid is not the product of any gluings described above; the addition of 6v7f\_2 to possible generators cannot contribute to its construction because they share the same number of faces.
\end{itemize}
\end{proof}

We observe that both the hexagonal pyramid and 6v7f\_2 are both half of the hexagonal bipyramid, sliced in two different ways; as such, they belong to the same commensurability class. 

We can further distinguish integral polyhedra by studying the stronger condition of superintegrality.

\subsubsection{Integral-Superintegral Polyhedra}
Of the four known seed polyhedra, two are also superintegral. The tetrahedron and square pyramid (as well as the other documented superintegral polyhedra) can be proved superintegral by an extension of \lemref{bendint}.

Aside from the polyhedra that have been documented, a theorem from \cite{KN18} guarantees the existence of additional superintegral polyhedra.

\begin{thm} \label{supersubset}
    Let $A$ be a superintegral polyhedron. If $A'$ is obtained by performing valid gluing operations on $A$, then 
    \begin{center}
        $\widetilde\sP(A') \subset \widetilde\sP(A)$
    \end{center}
\end{thm}

\begin{proof}
Let $A = \{V_A, E_A, F_A\}$ be superintegral and $A' = \{V_{A'}, E_{A'},F_{A'}\}$ be the polyhedron obtained by gluing $B = \{V_B, E_B, F_B\}$ to $A$ along vertex $v$. By definition, $V_{A'} = (V_A \cup V_B) \backslash v$, but in particular each vertex in $V_B$ is obtained by action of $R_v$ on some vertex in $V_A$. Similarly, each face in $V_{A'}$ is either already in $V_A$ or the result of $R_v$ applied to a face in $V_{A}$. 

By \lemref{refrelation}, all reflections in $A'\backslash B$ can be rewritten as a composition of reflections in $A$, since all elements of $A'$ are simply reflections in $A$ applied to elements of $A$. 
\end{proof}

\begin{lemma} \label{refrelation}
Let $\hat v_1 = v_1 R_{v}$. Then a reflection about $\hat v_1$ is equivalent to a series of reflections about $v$ and $v_1$, in particular 
\begin{center}
    $R_{\hat v_1} = R_v R_{v_1} R_v$.
\end{center}
\end{lemma}
\begin{proof}
\begin{align*}
    R_{\hat v_1} &= R_{v_1 R_{v}} \\
    &= I + 2Q(v_1 R_{v})^T (v_1 R_{v}) \\
    &= I + 2Q[I + 2v^TvQ^T]v_1^Tv_1[I + 2Qv^Tv] \\
    &= I + 2Qv_1^Tv_1 + 4Qv^TvQv_1^Tv_1 + 4Qv_1^Tv_1Qv^Tv + 8Qv^TvQv_1^Tv_1Qv^Tv \\
    &= R_vR_{v_1}R_v.
\end{align*}
\end{proof}

By \thmref{superint}, the groups associated with the above described superintegral polyhedra are arithmetic. However, not all integral polyhedra are superintegral. As it turns out, the contrapositive of this theorem completely describes all known integral but not superintegral polyhedra. 
\subsubsection{Integral but not Superintegral Polyhedra}
Of the known integral seed polyhedra, the hexagonal pyramid and 6v7f\_2 are not superintegral. \lemref{bendint} does not apply to either case as rational entries are present in one or more bend matrix, so integrality is proved by conjugation of the bend matrices. 
Superintegrality can be disproved by extension of \lemref{qnonintpoly} 
as well as application of \thmref{superint}, as the polyhedra are not arithmetic by the criterion described in \cite{Vin67}.

\subsection{Nonintegral Polyhedra} 
Aside from the four seed polyhedra and their gluings, we have documented many more polyhedra which are not integral. These can be understood in two broad subcategories.

\subsubsection{Rational-Nonintegral Polyhedra}
Of the nonintegral polyhedra, some (7v8f\_9, 7v9f\_8) have exclusively rational packings: rather than all bends being integral, they are all rational. An intuitive step would be to apply \lemref{rescaling} and find a conformally equivalent integral packing, however these packings cannot be rescaled, a result of the following lemma. 

\begin{lemma} \label{qnonintpoly}
Let $\{B_1, \dots ,B_n\}$ be bend matrices (see \defref{bend matrix}) of $\Pi$ and $B$ be any product of $\{B_1, \dots ,B_n\}$. If there is an entry in $B^n$ whose denominator grows without bound as $n \rightarrow \infty$, then $\Pi$ cannot be integral.
\end{lemma}

\subsubsection{Nonrational-Nonintegral Polyhedra}
All polyhedra which do not fit in one of the previous categories can be proved nonintegral by \lemref{nonintpoly}.

\section{Bianchi Group Packings}\label{bianchi packings}

\begin{Def}[Bianchi group]\label{Bianchi}
A Bianchi group $Bi(m)$ is the set of matrices 
\begin{equation}
SL_2(\mathcal{O}_m) \rtimes \langle \tau \rangle,
\end{equation}
where $\mathcal{O}_m$ indicates the ring of integers $\Z[\sqrt{-m}]$, $m$ is a positive square-free integer, and $\tau$ is a second-order element that acts on $SL_2(\mathcal{O}_m)$ as complex conjugation \cite{Vin90,BM13}. 
\end{Def}

These groups can also be viewed as discrete groups of isometries. \cite{Bia92} found that $Bi(m)$ is reflective---meaning that it is generated by a finite set of reflections---for $m\le19, m\not\in\{14,17\}$ (\cite{Bia92}); this list is complete (\cite{BM13}). 

\cite{Vin72} contains an algorithm on general quadratic forms, which takes the integral automorphism group of a quadratic form and halts if the reflection group is finitely generated, producing its fundamental polyhedron. \cite{Mcl13} applied this algorithm to the reflective \emph{extended Bianchi groups} $\hat{Bi}(m)$---the maximal discrete extension of $Bi(m)$ (cf. \cite{All66,BM13}, see also \cite{Ruz90,Sha90}). To be discrete, the fundamental polyhedron produced in the case of a finitely generated reflection group must be a Coxeter polyhedron, and thus a Coxeter diagram can be drawn. \cite{Mcl13} provides the roots bounding the fundamental polyhedron obtained from Vinberg's algorithm for all reflective extended Bianchi groups.\footnote{Corrections for errors in \cite{Mcl13}'s listings of roots can be found in \secref{mcl fix}.}

\subsection{Determining the clusters}

We computed the Gram matrix for each $\hat{Bi}(m)$.
We then iterated through the Gram matrix of each $\hat{Bi}(m)$ to identify all existing clusters and subgroups thereof.\footnote{We excluded all clusters wherein two vertices were orthogonal to each other due to redundancy in crystallographic circle packings.} By \thmref{structthm}, a Gram matrix can be used to identify clusters by identifying those rows whose entries $G_i$ exclusively satisfy the condition $\abs{G_i} \ge 1$ or $G_i = 0$. 

Every cluster $\cC$ within each $\hat{Bi}(m)$ was then used to produce a crystallographic packing by applying \thmref{structthm} via the methods outlined in \secref{producing cryst packing} above.

\begin{figure}[H]
\begin{minipage}{.5\textwidth}
    \[
      \begin{pmatrix}
      0 & 0 & -1 & 0 \\
      8 & 6 & 3 & 2\sqrt{10}
      \end{pmatrix}
    \]
    \caption*{(a)}
\end{minipage}
\begin{minipage}{.5\textwidth}
    \[
      \begin{pmatrix}
      1 & 0 & 1 & 0 \\
      0 & 0 & 0 & -1 \\
      \sqrt{10} & 0 & 0 & 1 \\
      -1 & 1 & 0 & 0 \\
      2\sqrt{2} & \sqrt{2} & 0 & \sqrt{5} \\
      3\sqrt{10} & 3\sqrt{10} & \sqrt{10} & 9 \\
      4\sqrt{10} & 4\sqrt{10} & 2\sqrt{10} & 11
      \end{pmatrix}
    \]
      \caption*{(b)}
\end{minipage}
\vspace*{\floatsep}
\begin{minipage}{.5\textwidth}
    \begin{tikzpicture}[scale=1.9]
    \coordinate (nine) at (0.5,0);
    \coordinate (eight) at (1.5, 0);
    \coordinate (seven) at (2,0);
    \coordinate (five) at (3,0);
    \coordinate (six) at (4,0.5);
    \coordinate (three) at (0.5,1);
    \coordinate (four) at (1.5,1);
    \coordinate (one) at (2,1);
    \coordinate (two) at (3,1);
    
    \draw[thick,double distance=2pt] (six) -- (two);
    \draw[thick,double distance=2pt] (six) -- (five);
    \draw[thick] (two) -- (five);
    \draw[ultra thick] (one) -- (two);
    \draw[ultra thick] (five) -- (seven);
    \draw[ultra thick] (three) -- (four);
    \draw[ultra thick] (eight) -- (nine);
    \draw[dashed] (one) -- (seven);
    \draw[dashed] (one) -- (eight);
    \draw[dashed] (one) -- (nine);
    \draw[dashed] (two) -- (eight);
    \draw[dashed] (three) -- (six);
    \draw[dashed] (three) -- (seven);
    \draw[dashed] (three) -- (eight);
    \draw[dashed] (three) -- (nine);
    \draw[dashed] (four) -- (five);
    \draw[dashed] (four) -- (seven);
    \draw[dashed] (four) -- (eight);
    \draw[dashed] (four) -- (nine);
    \draw[dashed] (six) -- (nine);
    
{    \filldraw[fill=blue] (one) circle (2pt);}
{    \filldraw[fill=red] (two) circle (2pt);}
{    \filldraw[fill=red] (three) circle (2pt);}
{    \filldraw[fill=red] (four) circle (2pt);}
{    \filldraw[fill=red] (five) circle (2pt);}
{    \filldraw[fill=red] (six) circle (2pt);}
{    \filldraw[fill=blue] (seven) circle (2pt);}
{    \filldraw[fill=red] (eight) circle (2pt);}
{    \filldraw[fill=red] (nine) circle (2pt);}
\end{tikzpicture}
    \caption*{(c)}
\end{minipage}
\begin{minipage}{.4\textwidth}
    \includegraphics[scale=.5]{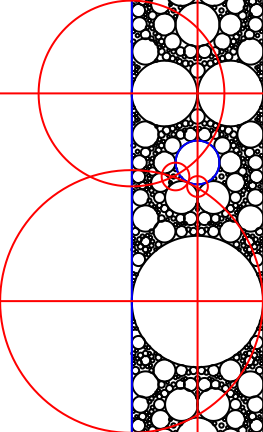}
    \caption*{(d)}
\end{minipage}
\caption{(a) Coordinates of cluster $\{1,7\}$ for $\hat{Bi}(10)$. (b) Corresponding cocluster coordinates. (c) Coxeter diagram for $\hat{Bi}(10)$, with cluster \{1,7\} highlighted in blue. (d) Diagram of $\hat{Bi}(10)$ cluster \{1,7\} packing, cluster in blue, with cocluster in red.}
\end{figure}

\subsection{Results}

All crystallographic packings that arise from the extended Bianchi groups have been documented, namely all packings from $\hat{Bi}(m)$ for $m=$ 1, 2, 5, 6, 7, 10, 11, 13, 14, 15, 17, 19, 21, 30, 33, 39, each with their corresponding Coxeter diagram, inversive coordinates matrix, Gram matrix, and diagrams of all possible cluster packings.\footnote{Note that $\hat{Bi}(3)$ as recorded in \cite{BM13} does not yield a crystallographic packing. See \secref{bi3fix} for treatment of $\hat{Bi}(3)$.}

All extended Bianchi group packings are determined to be arithmetic. We then classified all integral and non-integral extended Bianchi group packings.\footnote{Note that this is up to rescaling via circle inversion.} There are 145 integral $\hat{Bi}(m)$ packings and 224 non-integral $\hat{Bi}(m)$ packings, a complete list of which can be found in \secref{int and nonint bianchi}. Our first step was to compute the orbit of each packing up to some generation, and empirically conjecture whether or not the packing would be integral based on the bends produced. We then rigorously proved such conjectures, as described in the following sections.

\begin{figure}[H]
\centering
\begin{subfigure}{.5\textwidth}
  \centering
  \includegraphics[scale=.5]{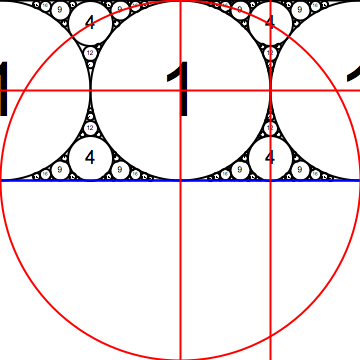}
  \captionof*{figure}{(a)}
  \label{Integral Bi diagram}
\end{subfigure}
\begin{subfigure}{.5\textwidth}
  \centering
  \[
  \begin{pmatrix}
  1 & 0 & 0 & 0 \\
  0 & 1 & 0 & 0 \\
  0 & 0 & 1 & 0 \\
  2 & 2 & 2 & -1
  \end{pmatrix}
  \cdot
  \begin{pmatrix}
  0 & 0 & 0 & -1 \\
  2 & 0 & 0 & 1 \\
  0 & 2 & 0 & 1 \\
  2 & 2 & 2 & 1
  \end{pmatrix}
  \]
  \caption*{(b)}
\end{subfigure}
\caption{(a) Integral packing generated from $\hat{Bi}(1)$ cluster \{3\}, depicted in blue, along with cocluster, depicted in red. Numbers indicate bends.
(b) $BV$ for $\hat{Bi}(1)$ cluster \{3\} and its orbit, demonstrating integrality of the packing.}
 \label{bi int by bendint}
\end{figure}

\subsubsection{Integral Bianchi group packings}
To prove integrality for Bianchi group packings, \lemref{rescaling} is first applied when necessary to rescale inversive coordinates of the clusters to integrality. 
\par In the cases where the associated bend matrices $B_i$ for a given Bianchi packing are integral, integrality is immediately proven via \lemref{bendint}. (See Figure \ref{bi int by bendint} for an example.) In the cases where the associated $B_i$'s are not integral, integrality can still be proven through \lemref{bendint} by first calling upon the following lemma, the proof of which relies on the integrality of all right-acting reflection matrices $R$ associated with each $\hat{Bi}(m)$ packing.\footnote{As proven in \cite{BM13}, there exist finitely-many $\hat{Bi}(m)$ packings. We have generated reflection matrices $R$ for all such packings, and determined that every associated $R$ matrix is integral. Note that an alternative proof of integrality for Bianchi groups whose bend matrices are not integral is implied by \cite{KN18}'s discussion of arithmeticity of the supergroup of a superintegral packing.}

\begin{lemma}\label{bounded B}
The product of any bend matrices of an integral Bianchi group packing is ``bounded rational."
\end{lemma}

\begin{proof}
By (\ref{BV}), for $V$ a $(n+2)\times(n+2)$ full-rank matrix of inversive coordinates, the right-acting reflection matrix $R$ can be expressed as
\begin{equation}\label{R as V inv}
    R=V^{-1} B V.
\end{equation}
Since $B$ is simply a left-acting reflection matrix, action of $B_i$ on $B_j V$ is simply further reflection of $V$, and hence $B_i B_j$ is a proper bend matrix. Similarly, $R$ can be broken into a series of reflection matrices $R_1 \cdots R_i$. Thus, we have
\begin{equation}
   V  R_1 \cdots R_i V^{-1} = B_i\cdots B_1.
\end{equation}
Suppose that $B' = B_i \cdots B_1$ has unbounded denominators. Since $V,V^{-1}$ are fixed matrices, this would imply that $R' = R_1 \cdots R_i$ has unbounded denominators as well. However, as each $R_j$ is known to be integral, $R'$ is definitely integral. Therefore, $B'$ cannot have unbounded denominators.
\end{proof}

\lemref{bounded B} implies that it is possible to clear the denominators of any $B'$ through some rescaling of $V$. Because there are no irrational entries in the integral Bianchi group bend matrices, and all bends of the cluster can be rescaled to integrality by \lemref{rescaling}, this proves that all bends produced by $B'V$ will then be integral.

\begin{figure}[H]
\centering
\begin{subfigure}{.5\textwidth}
  \centering
  \includegraphics[scale=.5]{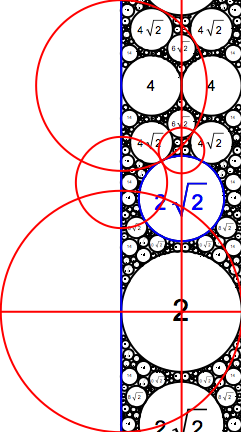}
  \captionof*{figure}{(a)}
  \label{Non-int Bi diagram}
\end{subfigure}
\begin{subfigure}{.5\textwidth}
  \centering
 \[
\begingroup
\renewcommand*{\arraystretch}{1.5}
 \begin{pmatrix}
 1 & 0 & 0 & 0 \\
 3\sqrt{2} & \frac{5}{2} & 2\sqrt{2} & -\half \\
 6 & \frac{3\sqrt{2}}{2} & 5 & -\frac{\sqrt{2}}{2} \\
 45\sqrt{2} & \frac{45}{2} & 30\sqrt{2} & -\frac{13}{2}
 \end{pmatrix}
\endgroup
 \]
 \caption*{(b)}
\end{subfigure}
\caption{(a) Nonintegral packing generated from $\hat{Bi}(14)$ cluster \{1,8\}, depicted in blue, along with cocluster, depicted in red. Numbers indicate bends. Note mixture of integral and nonintegral bends, indicating that this packing cannot be rescaled to integrality.
(b) One of the associated bend matrices $B$. Note the existence of a nonlinear relation in each row except the first.}
\end{figure}

\subsection{Nonintegral Bianchi group packings}
While all Bianchi groups give rise to integral packings, there exist packings whose cluster circles cannot be rescaled to integrality which will produce nonintegral packings. For instance, the packing produced by $\hat{Bi}(17)$ with the cluster of vertices $\{8,13\}$ has both a bend of $\sqrt{34}$ (vertex 8) and a bend of $\sqrt{17}$ (vertex 13); clearly this cluster cannot be rescaled in such a way that clears both bends to integrality. 

To prove nonintegrality of such packings, we built an over-determined matrix $V$ from the inversive coordinates of the cluster, with supplementary inversive coordinates from the orbit added if necessary. We then applied \lemref{nonintpoly}.
In every non-integral packing, the solution set for $g$ contained irrational coefficients, indicating an inherent nonintegral relation between the entries of the bend matrices for the packing. (See \secref{nonint bianchi proof example} for an example.)

\section{Higher Dimensional Packings}\label{hd}

We now seek to apply \thmref{structthm} to the finitely-many commensurability classes guaranteed by \thmref{finite}, as these classes are known to exist but are not guaranteed to admit packings. There are a number of techniques available to produce such candidates. We note that all packings produced through these techniques are arithmetic, and in doing so we show that the quadratic forms in question are commensurate with a supergroup of a packing. 

We consider for fixed $d$ and $n$ the configuration of inversive coordinates $\nye{V}=\{v_i\}_{i=1}^m$ and Gram matrix $G=\{g_{ij}\}_{i,j=1}^m$. (More generally, we'll consider $i,j$ instead ranging on some explicit index set.) In \thmref{validatedoubling}, we show an important result that validates a technique (doubling, see \defref{doublingdef}) that contributes to producing the desired packings. 

\begin{Def}
For $1\le a,b\le m$, we write $b.a$, ``the action of $b$ on $a$" or ``$b$ acting on $a$," to denote $v_aR_{v_b}$. 
\end{Def}

Note that this action is right-associative, i.e. $a.b.c=a.(b.c)$. 

\begin{lemma}
$a.a.b=b$, as inversion is an involution. 
\end{lemma}

\begin{lemma}
$(a.b).c=a.b.a.c$. 
\end{lemma}
\begin{proof}
This is \lemref{refrelation} under different notation. 
\end{proof}

\begin{Def}[reversing orientation]
Let $\nye{a}$ denote the oriented hypersphere $a$ with reversed orientation, i.e. $a.a$. 
\end{Def}

\subsection{Doubling}

\begin{Def}[doubling]\label{doublingdef}
Let $\nye{V}_j=\nye{V}\backslash\{v_j\}$ for $1\le j\le m$. We say that we \emph{double} $\nye{V}$ about $j$ when we compute $\nye{V}^j=\nye{V}_j\cup j.\nye{V}_j$. 
\end{Def}

Doubling can be thought of as a ``hyperbolic gluing" in the same vein as \defref{gluingop} for Euclidean polyhedra, wherein a configuration is extended into $\H^{n+1}$ through Poincar\'{e} extension and then doubled about a face. 

\begin{thm}\label{validatedoubling}
$\<\nye{V}^j\><\<\nye{V}\>$ and $\left[\<\nye{V}\>:\<\nye{V}^j\>\right]\le\infty$. 
\end{thm}
\begin{proof}
Without loss of generality, let $j=1$.

$\<\nye{V}^1\><\<\nye{V}\>$ follows immediately as $\<\nye{V}^1\>=\<2,\dots,m,1.2,\dots,1.m\>$ consists exclusively of elements of $\<\nye{V}\>$. 

We aim to show that if $\nye{V}$ extends through Poincar\'{e} extension to a $\R^n$-bounded $\H^{n+1}$ polytope $P$ that touches $\R^n$ at finitely many cusps, then doubling about $1$ also yields such a polytope $P^1$. This will prove the result because the quotients of the volumes equals the index -- $\left[P^1\right]/\left[P\right]=\left[\<\nye{V}\>:\<\nye{V}^1\>\right]$. This is because each coset represents a ``copy" of $P$ that can be mapped into $P^1$ through members of $\<\nye{V}\>$, i.e. $P$ can be thought as the gluing-together of several copies of $P^1$, and specifically $\left[\<\nye{V}\>:\<\nye{V}^1\>\right]$-many. 

Let $j^+$ denote the interior of $v_j$ and let $j^-$ denote the interior of $\nye{j}$. Further, let $\sqcap_1=\bigcap\limits_{j=2}^mj^+$. The interior of $\nye{V}^1$ is 
\begin{align}
\sqcap_1\cap1.\sqcap_1=&(\sqcap_1\cap(1^+\cap1.\sqcap_1))\nonumber\\
            &\cup(\sqcap_1\cap(1\cap1.\sqcap_1))\nonumber\\
            &\cup(\sqcap_1\cap(1^-\cap1.\sqcap_1)),\label{intersection overall}\\
\sqcap_1\cap(1^+\cap1.\sqcap_1)=&((1^+\cap\sqcap_1)\cap(1^+\cap1.\sqcap_1))\nonumber\\
            &\cup((1\cap\sqcap_1)\cap(1^+\cap1.\sqcap_1))\nonumber\\
            &\cup((1^-\cap\sqcap_1)\cap(1^+\cap1.\sqcap_1)).\label{intersection part}
\end{align}

From the hypothesis, we have that $(1^+\cup1)\cap\sqcap_1$ is finite, and we further know that $1^-\cap1^+=\emptyset$ and is thus finite; thus, $\sqcap_1\cap1.\sqcap_1$ is the finite union of finite sets and is finite. 

Similarly, we can prove boundedness: we know that $(1^+\cup1)\cap\sqcup_1$ is bounded, so substituting \eqref{intersection part} into \eqref{intersection overall} as was implicitly done just above gets 

\begin{align}
\sqcap_1\cap1.\sqcap_1=&((1^+\cap\sqcap_1)\cap(1^+\cap1.\sqcap_1))\nonumber\\
            &\cup((1\cap\sqcap_1)\cap(1^+\cap1.\sqcap_1))\nonumber\\
            &\cup((1^-\cap\sqcap_1)\cap(1^+\cap1.\sqcap_1))\nonumber\\
            &\cup(\sqcap_1\cap(1\cap1.\sqcap_1))\nonumber\\
            &\cup(\sqcap_1\cap(1^-\cap1.\sqcap_1))
\end{align}

\noindent which is the finite union of bounded and empty sets, and hence is bounded. 
\end{proof}

Doubling, in certain cases, creates configurations that have clusters. However, this requires that the resultant configuration exclusively has angles of the form $\frac{\pi}{n}$, in order for there to exist a Coxeter diagram. Practically speaking: 

\begin{lemma}\label{whatcandouble}
If doubling about 1 generates a reflective group, 
then in $\nye{V}$'s Coxeter diagram, 1 is joined to other nodes exclusively by an even number of edges. (A thick line is considered to be an even number of edges.)
\end{lemma}

\lemref{whatcandouble} is used to remove mirrors in a configuration from consideration for doubling. Of course, it is also necessary that the resultant configuration generates a group of mirrors commensurate to the original group. This is the primary function of \thmref{validatedoubling}. 

The principle of doubling was successfully applied to obtain clusters, and hence commensurability classes yielding packings, described in full in \secref{pack13} and \secref{pack35}. 

\subsection{Beyond Doubling}\label{beyond}

Unfortunately, doubling is not a panacea, and there remain instances that are left unresolved by doubling. Here we describe instances where the Coxeter diagram obtained by Vinberg's algorithm, either in \cite{Vin72} or \cite{Mcl10}, lack a cluster, as do all of their doublings, but there exist other subgroups that admit packings through their diagrams. We begin our siege on these other diagrams with a lemma useful in proofs to come. 

\begin{lemma}\label{circint}
The $n$-dimensional oriented hypersphere specified by $\left(\hat{b},b,bz\right)$ has interior given by $x=(x_i)_{i=1}^n\in\hat{\R^n}$\footnote{In the case of $x=\infty$ we can safely view $x$ as the $n$-tuple with each entry equal to $\infty$.} satisfying $$0<
\begin{cases}
bz\cdot x-\half\hat{b}&b=0\\
\left(\frac{1}{b^2}-\abs{z-x}^2\right)\sign b&b\neq0
\end{cases}$$ where $v\cdot w=vI_nw^T$, i.e. the ``standard" dot product. 
\end{lemma}
\begin{proof}
\textbf{Case $b=0$.} We know that the hyperplane has normal vector given by $\hat{n}=bz$, from \lemref{planefacts}. We first consider the case $\hat{b}=0$, i.e. the codimension-1 hyperplane passes through the origin (by \lemref{planefacts}, as in that context we would have $\abs{P}=0\implies P=0$). The half-space ``on the same side as" $\hat{n}$ is characterized by $\hat{n}\cdot x>0$. Consider now the case $\hat{b}\neq0$. Consider a translation of space sending $P\mapsto0$ and $x\mapsto x'$. Then it is clear that we must have $\hat{n}\cdot x'>0$ in this mapping of space (as the plane now passes through the origin). Of course, this mapping is simply defined as $x\mapsto x'=x-P$. $P$ can be computed as $\half\hat{b}\hat{n}$, since $P$ is the nearest point to the origin and hence the segment connecting 0 and $P$ is perpendicular to the plane, thus parallel to $\hat{n}$. Therefore, we have $$0<\hat{n}\cdot\left(x-\half\hat{b}\hat{n}\right)=\hat{n}\cdot x-\half\hat{b}\hat{n}\cdot\hat{n}=bz\cdot x-\half\hat{b}.$$

\textbf{Case $b\neq0$.} If $b>0$ then we are looking for the standard notion of the interior of a hypersphere, i.e. the points within $r$ of the center. Namely, $$r=\frac{1}{b}>\abs{z-x}>0\implies\frac{1}{b^2}>\abs{z-x}^2\implies\frac{1}{b^2}-\abs{z-x}^2>0.$$ If $b<0$ then we are looking for the complement of $\overline{B_z\left(1/\abs{b}\right)}$ (the closed ball centered at $z$ with radius $1/\abs{b}$), i.e. $x$ must satisfy $$0<r=-\frac{1}{b}<\abs{z-x}\implies\frac{1}{b^2}<\abs{z-x}^2\implies\frac{1}{b^2}-\abs{z-x}^2<0.$$ We now see that both cases of $b$'s sign are captured by $\left(\frac{1}{b^2}-\abs{z-x}^2\right)\sign b>0$.
\end{proof}

\begin{Def}
In the context of $\nye{V}=\{v_i\}_{i=1}^m$ a configuration of inversive coordinates generating $\<\nye{V}\>\le\isom(\H^{n+1})$ of finite index where $\H^{n+1}$ is viewed as arising from $-dx_0^2+\sum\limits_{i=1}^nx_i^2$, and for $1\le j\le m$, we write $(d.n.j)$ to denote the application of \lemref{circint} to $v_j$. 
\end{Def}

We now give an example of a fruitful result that does not rely on doubling. 

\subsubsection{$d=3,n=3$}\label{pack33}

We obtain the following Coxeter diagram

\begin{center}
\begin{tikzpicture}
\coordinate (2) at (0,0);
\coordinate (4) at (2,0);
\coordinate (7) at (1,2);
\coordinate (6) at (1,1);
\coordinate (1) at (2,1);
\coordinate (3) at (3,1);
\coordinate (5) at (2.5,2);

\draw[ultra thick] (2) -- (7);
\draw[ultra thick] (1) -- (3);
\draw[thick, double distance=2pt] (1) -- (5);
\draw[thick, double distance=2pt] (3) -- (5);
\draw[dashed] (1) -- (6);
\draw[dashed] (2) -- (4);
\draw[dashed] (2) -- (6);
\draw[dashed] (3) -- (4);
\draw[dashed] (5) -- (7);

{    \filldraw[fill=white] (1) circle (2pt) node[below] {\scriptsize 5};}
{    \filldraw[fill=red] (2) circle (2pt) node[below] {\scriptsize 6};}
{    \filldraw[fill=white] (3) circle (2pt) node[right] {\scriptsize 7};}
{    \filldraw[fill=red] (4) circle (2pt) node[below] {\scriptsize 8};}
{    \filldraw[fill=white] (5) circle (2pt) node[above] {\scriptsize 9};}
{    \filldraw[fill=red] (6) circle (2pt) node[above] {\scriptsize 10};}
{    \filldraw[fill=red] (7) circle (2pt) node[above] {\scriptsize 11};}

\end{tikzpicture}
\end{center}

\noindent arising from the configuration 

\begin{align}
\begin{array}{c|cccc|l}\label{coords33}
 & \hat{b} & b & bx & by & \text{def'd as:}\\\hline
5& -2 & 0 & -\frac{\sqrt{3}}{2} & -\frac{1}{2} & \nye{3} \\
6& 0 & 0 & -\frac{1}{2} & \frac{\sqrt{3}}{2}& 1.2 \\
7& 6 & 0 & \frac{\sqrt{3}}{2} & \frac{1}{2}& 3.1.2.1.2.1.3 \\
8& \sqrt{2} & \frac{\sqrt{2}}{2} & \frac{\sqrt{6}}{2} & -\frac{\sqrt{2}}{2} & 1.3.4\\
9& \sqrt{2} & \frac{\sqrt{2}}{2} & \frac{\sqrt{6}}{2} & \frac{\sqrt{2}}{2} & 3.4\\
10& 5 \sqrt{2} & \frac{\sqrt{2}}{2} & \sqrt{6} & 0 & 3.1.2.1.3.4\\
11& 2 \sqrt{3} & 0 & \frac{1}{2} & -\frac{\sqrt{3}}{2} & 1.2.3.1.2
\end{array}
\end{align}

\begin{lemma}
\eqref{coords33} has empty interior in $\R^2$, and extends by Poincar\'{e} extension to a hyperbolic polytope of finite volume. 
\end{lemma}
\begin{proof}
We first show that the configuration has empty interior, by supposing towards contradiction that $x_1\ge0$: 

\begin{align*}
1-\frac{\sqrt{3}}{2}x_1-\half x_2&>0\tag{3.3.5}\\
x_2>2+\sqrt{3}x_1&\ge2\\
\label{ineq3310} 2-(2\sqrt{3}-x_1)^2-x_2^2&>0\tag{3.3.10}\\
2>(2\sqrt{3}-x_1)^2+x_2^2&\ge x_2^2\\
&\ge2^2,
\end{align*}

\noindent a contradiction. Hence in the supposed mutual interiors of the configuration, $x_1<0$. However, again following \eqref{ineq3310}, we find 

\begin{align*}
2>(2\sqrt{3}-x_1)^2+x_2^2&\ge (2\sqrt{3}-x_1)^2\\
&>(2\sqrt{3})^2,
\end{align*}

\noindent a further contradiction. Hence no point $(x_1,x_2)\in\R^2$ lies in the mutual interior of the specified configuration. 

\eqref{ineq3310} also gives bounds for each coordinate: $2\sqrt{3}-\sqrt{2}<x_1<2\sqrt{3}+\sqrt{2}$ and $\abs{x_2}<\sqrt{2}$. Since the intersection of the respective Poincar\'{e} extensions of the circles is bounded and does not meet the boundary of $\H^3$, it must be of finite volume. 
\end{proof}

This admits a packing through e.g. $\{6\}$, the cluster consisting just of the vector $\begin{pmatrix} 0 & 0 & -\frac{1}{2} & \frac{\sqrt{3}}{2} \end{pmatrix}$. This was uncovered by analyzing by hand the orbit of $\{1,2,3,4\}$ acting on itself. 

\begin{figure}[H]
    \centering
    \includegraphics[scale=.5]{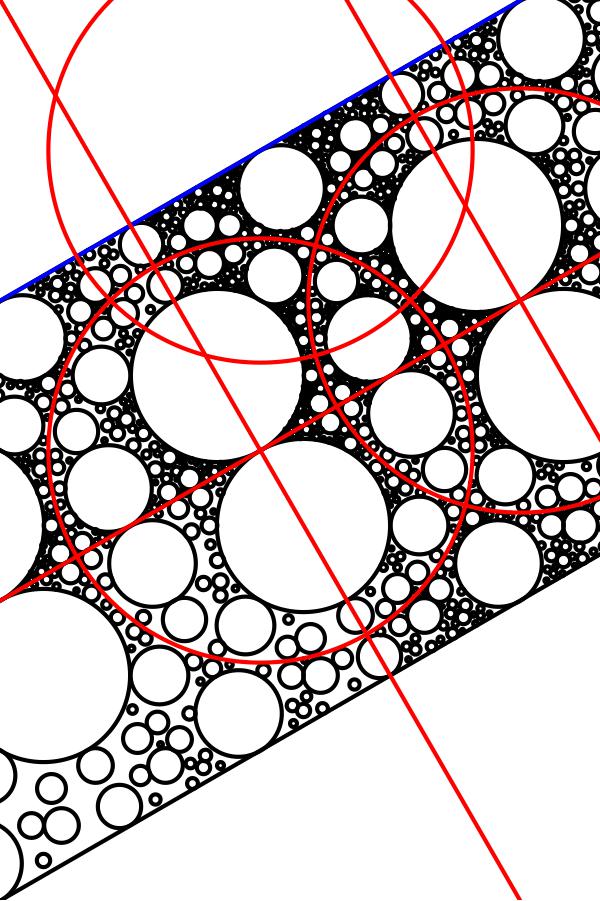}
    \caption{A packing arising from cluster $\{6\}$ in \secref{pack33}.}
\end{figure}

\subsubsection{$\hat{Bi}(3)$}\label{bi3fix}

\cite{KN18} includes the following useful fact: 

\begin{lemma}\label{unfold}
The Coxeter diagram for $\hat{Bi}(3)$,
\begin{center}
\begin{tikzpicture}
\coordinate (one) at (0,0);
\coordinate (two) at (1,0);
\coordinate (three) at (2,0);
\coordinate (four) at (3,0);

\draw[thick, double distance=2.5pt] (one) -- (two);
\draw[thick, double distance=0.3pt] (one) -- (two);
\draw[thick] (two) -- (three);
\draw[thick] (three) -- (four);

{    \filldraw[fill=white] (one) circle (2pt) node[above] {\scriptsize 1};}
{    \filldraw[fill=white] (two) circle (2pt) node[above] {\scriptsize 2};}
{    \filldraw[fill=white] (three) circle (2pt) node[above] {\scriptsize 3};}
{    \filldraw[fill=white] (four) circle (2pt) node[above] {\scriptsize 4};}

\end{tikzpicture}
,\end{center}
admits the subgroup corresponding to 
\begin{center}
\begin{tikzpicture}
\coordinate (one) at (0,0);
\coordinate (two) at (1,0);
\coordinate (three) at (2,0);
\coordinate (four) at (3,0);
\coordinate (five) at (4,0);

\draw[ultra thick] (one) -- (two);
\draw[thick, double distance=2.5pt] (two) -- (three);
\draw[thick, double distance=0.3pt] (two) -- (three);
\draw[thick] (three) -- (four);
\draw[ultra thick] (four) -- (five);

{    \filldraw[fill=red] (one) circle (2pt) node[above] {\scriptsize $a$};}
{    \filldraw[fill=white] (two) circle (2pt) node[above] {\scriptsize 1};}
{    \filldraw[fill=white] (three) circle (2pt) node[above] {\scriptsize 2};}
{    \filldraw[fill=white] (four) circle (2pt) node[above] {\scriptsize 3.1};}
{    \filldraw[fill=red] (five) circle (2pt) node[above] {\scriptsize 4};}
\end{tikzpicture}
\end{center}
for $a=(((3.2).1).4).((3.2).1)=3.2 .3 .1 .3 .2 .3 .4 .3 .2 .3 .\nye{1}=(3.2.3.1).4.(3.2).\nye{1}$.
\end{lemma}

\begin{figure}[H]
    \centering
    \includegraphics[scale=.5]{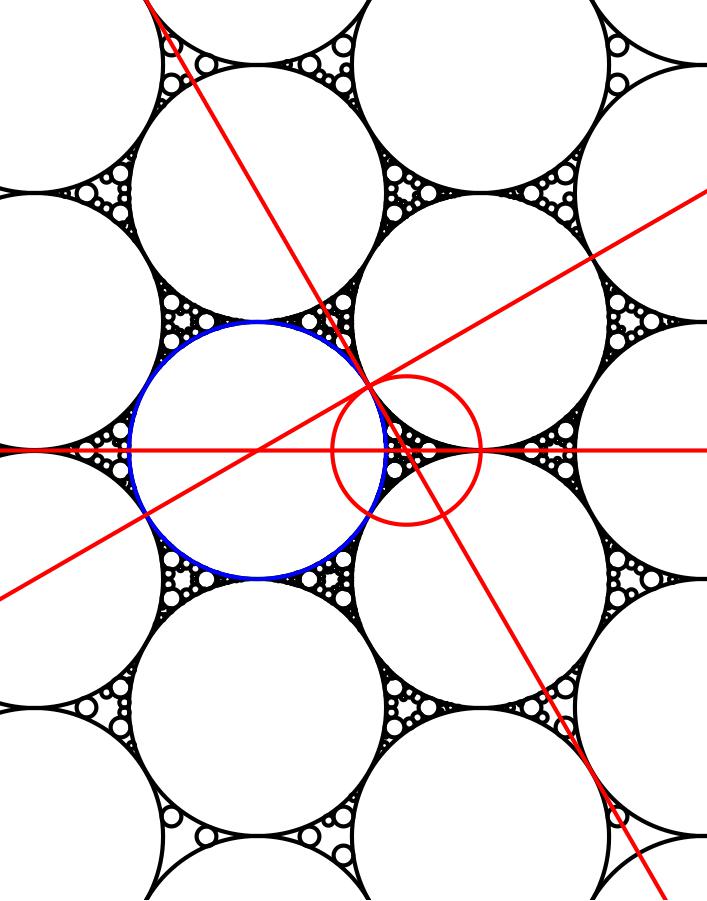}
    \caption{A packing arising from cluster $\{4\}$ in \secref{bi3fix}.}
\end{figure}

In each of the transformations in \secref{pack36}--\secref{pack313}, \lemref{unfold} was applied to a subdiagram having the form \begin{tikzpicture}
\coordinate (one) at (0,0);
\coordinate (two) at (1,0);
\coordinate (three) at (2,0);
\coordinate (four) at (3,0);

\draw[thick, double distance=2.5pt] (one) -- (two);
\draw[thick, double distance=0.3pt] (one) -- (two);
\draw[thick] (two) -- (three);
\draw[thick] (three) -- (four);

{    \filldraw[fill=white] (one) circle (2pt) node[above] {\scriptsize 1};}
{    \filldraw[fill=white] (two) circle (2pt) node[above] {\scriptsize 2};}
{    \filldraw[fill=white] (three) circle (2pt) node[above] {\scriptsize 3};}
{    \filldraw[fill=white] (four) circle (2pt) node[above] {\scriptsize 4};}

\end{tikzpicture} to obtain \begin{tikzpicture}
\coordinate (one) at (0,0);
\coordinate (two) at (1,0);
\coordinate (three) at (2,0);
\coordinate (four) at (3,0);
\coordinate (five) at (4,0);

\draw[ultra thick] (one) -- (two);
\draw[thick, double distance=2.5pt] (two) -- (three);
\draw[thick, double distance=0.3pt] (two) -- (three);
\draw[thick] (three) -- (four);
\draw[ultra thick] (four) -- (five);

{    \filldraw[fill=white] (one) circle (2pt) node[above] {\scriptsize $a$};}
{    \filldraw[fill=white] (two) circle (2pt) node[above] {\scriptsize 1};}
{    \filldraw[fill=white] (three) circle (2pt) node[above] {\scriptsize 2};}
{    \filldraw[fill=white] (four) circle (2pt) node[above] {\scriptsize 3.1};}
{    \filldraw[fill=white] (five) circle (2pt) node[above] {\scriptsize 4};}
\end{tikzpicture}. Another transformation was then applied to the remaining $m-4$ mirrors in the configuration such that the resultant configuration had a cluster and corresponded to a finite-volume $\H^{n+1}$ polytope, provably so by \lemref{circint}. 

For each configuration provided in \secref{hdpackingdata}, the numbering implicitly referenced comes from the configurations obtained through Vinberg's algorithm in \cite{Vin72} and \cite{Mcl10}, and listed at \url{math.rutgers.edu/~alexk/crystallographic}. 

\subsection{Unresolved questions}

This work has shown that for $1\le d\le 3$, the only cases not known to admit a commensurability class of packings are in $d=1$, $n>3$. Therefore the immediate next steps would be to consider those cases, as well as $d>3$. The data for $d>3$ can be found in \cite{Mcl13}. 

\appendix

\section{Integral Polyhedra}\label{integralpoly}
\subsection{Construction of Integral Polyhedra}
All known integral polyhedra which are not one of the four seed polyhedra can be constructed by gluings of seed polyhedra. 
For ease of notation, we let $t$ = tetrahedron and $s$ = square pyramid. $\mathscr{V}^+$ is a vertex gluing and $\mathscr{F}_n^+$ is a face gluing along an $n$-gon face. 
\begin{center}
    \begin{tabular}{ | l | l | l |}
    \hline
     Name & Planar Graph & Construction \\ \hline
    Triangular bipyramid & \includegraphics[width=2cm]{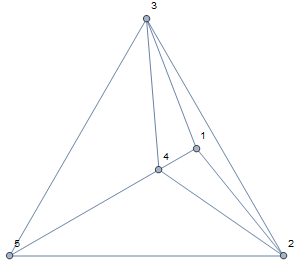} & $t \; \mathscr{F}^+ \; t$  \\ \hline
    6v8f\_1 & \includegraphics[width=1.8cm]{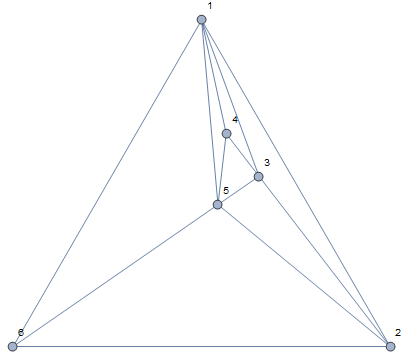} & $t\; \mathscr{F}^+ \; t \; \mathscr{F}^+ \;t$  \\ \hline
    Octahedron & \includegraphics[width=1.8cm]{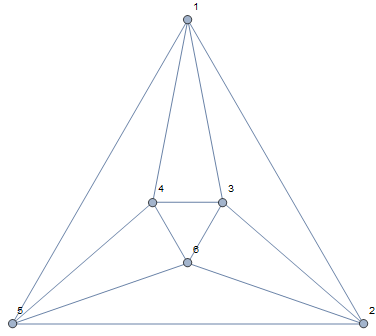} & $s \;\mathscr{F}_4^+\; s$ \\ \hline
    Elongated triangular pyramid & \includegraphics[width=1.8cm]{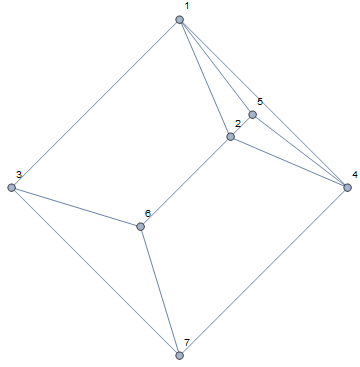} & $t \; \mathscr{V}^+ \; t \; \mathscr{F}^+ \; t$ \\ \hline
    7v8f\_6 & \includegraphics[width=1.8cm]{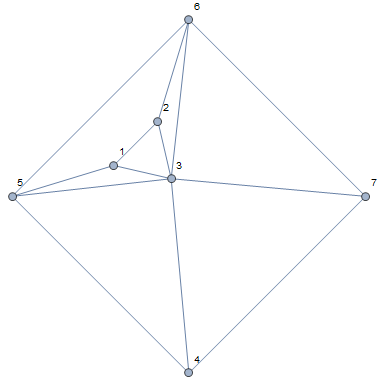} & $s \; \mathscr{F}_3^+ \; s$ \\ \hline
    7v8f\_7 & \includegraphics[width=1.8cm]{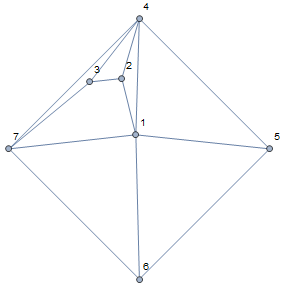} & $s \; \mathscr{F}_3^+ \; s$ \\ \hline
    7v10f\_1 & \includegraphics[width=1.8cm]{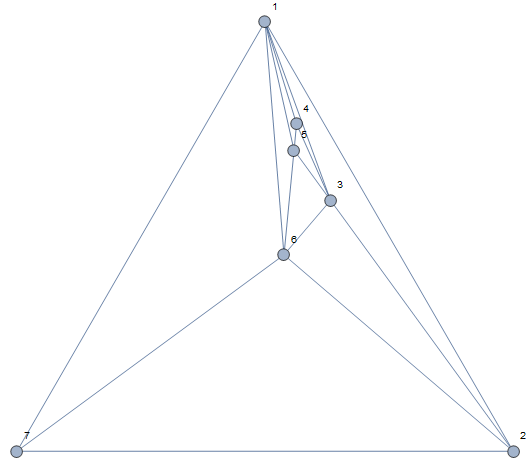} & $t \; \mathscr{F}^+ \; t \; \mathscr{F}^+ \; t$ \\ \hline
    7v10f\_2 & \includegraphics[width=1.8cm]{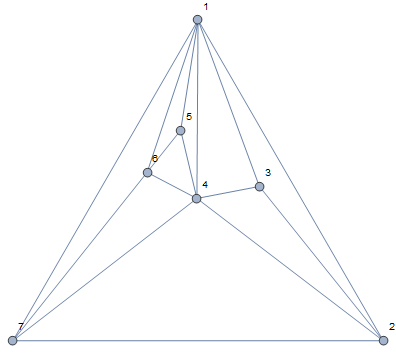} & $t \; \mathscr{F}^+ \; t \; \mathscr{F}^t \; t$\\ \hline
    7v10f\_3 & \includegraphics[width=1.8cm]{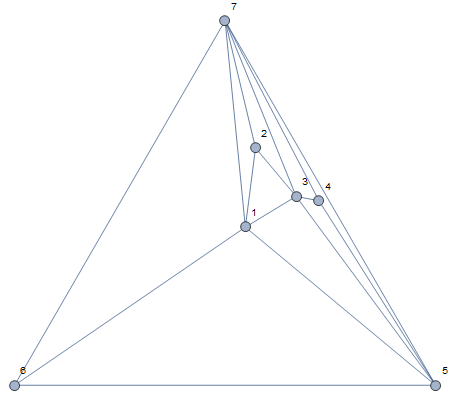} & $t \; \mathscr{F}^+ \; t \; \mathscr{F}^+ \; t$\\ \hline
    \end{tabular}
\end{center}

\subsection{Proving integrality} 
Most of the integral polyhedra which have been identified can be proven integral by \lemref{bendint}. The integral bend matrices associated to every such polyhedron can be found on our website. The only two which cannot be proven integral in this way are the hexagonal pyramid and 6v7f\_2. The bend matrices associated to these polyhedra are rational but not strictly integral, so we must verify that the fractional components of the bends can always be cleared by rescaling. This process is currently done in an ad hoc way by inspection of how the rational entries of a bend matrix change under multiplication with other bend matrices.

\section{Nonintegral Polyhedra}\label{nonintegralpoly}
All nonintegral polyhedra can be proved as such in one of two ways: all nonintegral-nonrational by \lemref{nonintpoly} and all nonintegral-rational by \lemref{qnonintpoly}. We give an example of each below. 

\subsubsection{Nonintegral-nonrational}

The following is the general form of any matrix in the cokernel of $V$ for the polyhedron 6v7f\_1.

\begin{center}
$\left(
\begin{array}{cccccc}
\alpha b_{12}-\beta b_{13} & b_{12} & b_{13} & \gamma b_{12}\delta b_{13} & \sqrt{2} \left(b_{12}+b_{13}\right)
   &\alpha \left(b_{12}+b_{13}\right) \\
\alpha b_{22}-\beta b_{23} & b_{22} & b_{23} & \gamma b_{22}\delta b_{23} & \sqrt{2} \left(b_{22}+b_{23}\right) &
  \alpha \left(b_{22}+b_{23}\right) \\
\alpha b_{32}-\beta b_{33} & b_{32} & b_{33} & \gamma b_{32}\delta b_{33} & \sqrt{2} \left(b_{32}+b_{33}\right) &
  \alpha \left(b_{32}+b_{33}\right) \\
\alpha b_{42}-\beta b_{43} & b_{42} & b_{43} & \gamma b_{42}\delta b_{43} & \sqrt{2} \left(b_{42}+b_{43}\right) &
  \alpha \left(b_{42}+b_{43}\right) \\
\alpha b_{52}-\beta b_{53} & b_{52} & b_{53} & \gamma b_{52}\delta b_{53} & \sqrt{2} \left(b_{52}+b_{53}\right) &
  \alpha \left(b_{52}+b_{53}\right) \\
\alpha b_{62}-\beta b_{63} & b_{62} & b_{63} & \gamma b_{62}\delta b_{63} & \sqrt{2} \left(b_{62}+b_{63}\right) &
  \alpha \left(b_{62}+b_{63}\right) \\
\end{array}
\right)$
\end{center}
$\alpha, \beta, \gamma, \delta$ are all irrational constants, allowing \lemref{nonintpoly} to be applied. 
\subsubsection{Nonintegral-Rational}
\begin{center}

7v9f\_8: $(4.9)^n$

$\left( 
\begin{array}{ccccccc}
 0 & 0 & \frac{25^{1-n}}{16}-\frac{25^n}{16} & 0 & 2\cdot 5^{-2 n}+\frac{25^n}{8}-\frac{31}{40} & -\frac{4}{5} & -\frac{7\cdot 5^{-2 n}}{16}+\frac{5^{2 n}}{16}+\frac{9}{40} \\
 0 & 0 & \frac{25^{1-n}}{16}-\frac{25^n}{16} & 0 & 2\cdot 5^{-2 n}+\frac{25^n}{8}-\frac{49}{40} & \frac{4}{5} & -\frac{7\cdot 5^{-2 n}}{16}+\frac{5^{2 n}}{16}-\frac{9}{40} \\
 0 & 0 & \frac{25^{1-n}}{18}-\frac{7\cdot 25^n}{18} & 0 & \frac{16\cdot 25^{-n}}{9}+\frac{7\cdot 25^n}{9}-\frac{23}{9} & 0 & \frac{7\cdot 25^n}{18}-\frac{7\cdot 25^{-n}}{18} \\
 0 & 0 & \frac{25^{1-n}}{64}-\frac{25^{n+1}}{64} & 0 & \frac{25^{-n}}{2}+\frac{25^{n+1}}{32}-\frac{23}{32} & -1 & -\frac{7\cdot 25^{-n}}{64}+\frac{25^{n+1}}{64}+\frac{9}{32} \\
 0 & 0 & 0 & 0 & 1 & 0 & 0 \\
 0 & 0 & \frac{25^{1-n}}{64}-\frac{25^{n+1}}{64} & 0 & \frac{5^{-2 n}}{2}+\frac{25^{n+1}}{32}-\frac{41}{32} & 1 & -\frac{7\cdot 25^{-n}}{64}+\frac{25^{n+1}}{64}-\frac{9}{32} \\
 0 & 0 & \frac{25^{1-n}}{18}-\frac{25^{n+1}}{18} & 0 & \frac{16\cdot5^{-2 n}}{9}+\frac{25^{n+1}}{9}-\frac{41}{9} & 0 & \frac{25^{n+1}}{18}-\frac{7\cdot25^{-n}}{18} \\
\end{array}
\right)$

\end{center}
The above matrix is the general form of $(4.9)^n$, where $4.9$ is the product of bend matrices associated with the 4th and 9th faces in the polyhedron 7v9f\_8. It has a number of entries which have denominators in the form $c^n$ and are thus unbounded.

\section{Corrections to \cite{Mcl13}}
\label{mcl fix}

The following corrections are stated in reference to Appendix F in \cite{Mcl13}:

\begin{itemize}
    \item In table F.2, vector $e_4$ should be $(2,0,0,-1)$, not $(1,0,0,-1)$.
    \item In table F.3, vector $e_4$ should be $(-1,1,0,0)$, not $(1,1,0,0)$.
    \item In table F.9, the self-product $(e,e)$ of $e_8$ should be 2, not 26.
    \item In table F.16, $e_3$ should be $(0,0,0,1)$ not $(0,0,1,-2)$; similarly $e_4$ should be $(33,0,0,1)$, not $(33,0,1,-2)$, since $m \equiv 1 \text{ (mod }4)$.
    \item In table F.17, vector $e_4$ should be $(39,0,-1,2)$, not $(33,0,-1,2)$.
The self-product $(e,e)$ of $e_3$ should be 78, not 66; similarly the self-product $(e,e)$ of $e_4$ should be 78, not 66.
\end{itemize}

\section{Integral and non-integral Bianchi packings}
\label{int and nonint bianchi}

The following is a complete list of all integral (145) and non-integral (224) crystallographic packings that arise from the extended Bianchi groups, referred to here as Bi($m$):\\

Integral:
\begin{itemize}
    \item Bi(1) : \{1\}, \{3\}
   
    \item Bi(2) : \{1\}, \{3\}
    
    \item Bi(5) : \{3\}, \{4\}, \{3,4\}
   
    \item Bi(6) : \{1\}, \{3\}, \{4\}, \{3,4\}
    
    \item Bi(7) : \{3\}, \{4\}, \{3,4\} 
    \item Bi(10) : \{1\}, \{3\}, \{4\}, \{7\}, \{8\}, \{9\}, \{1,7\}, \{3,4\}, \{3,8\}, \{3,9\}, \{4,8\}, \{4,9\}, \{8,9\}, \{3,4,8\}, \{3,4,9\}, \{3,8,9\}, \{4,8,9\}, \{3,4,8,9\} 
    \item Bi(11) : \{3\}, \{4\}, \{3,4\} 
    \item Bi(13) : \{3\}, \{4\}, \{9\}, \{10\}, \{3,4\}, \{3,9\}, \{3,10\}, \{4,9\}, \{4,10\}, \{9,10\}, \{3,4,9\}, \{3,4,10\}, \{3,9,10\}, \{4,9,10\}, \{3,4,9,10\} 
    \item Bi(14) : \{1\}, \{3\}, \{4\}, \{7\}, \{8\}, \{9\}, \{3,4\}, \{7,9\} 
    \item Bi(15) : \{3\}, \{4\}, \{7\}, \{8\}, \{3,4\}, \{3,7\}, \{3,8\}, \{4,7\}, \{4,8\}, \{7,8\}, \{3,4,7\}, \{3,4,8\}, \{3,7,8\}, \{4,7,8\}, \{3,4,7,8\} 
    \item Bi(17) : \{3\}, \{4\}, \{8\}, \{11\}, \{12\}, \{13\}, \{3,4\}, \{3,12\}, \{3,13\}, \{4,12\}, \{4,13\}, \{8,11\}, \{12,13\}, \{3,4,12\}, \{3,4,13\}, \{3,12,13\}, \{4,12,13\}, \{3,4,12,13\} 
    \item Bi(19) : \{3\}, \{4\}, \{3,4\} 
    \item Bi(21) : \{3\}, \{4\}, \{9\}, \{11\}, \{3,4\}, \{3,9\}, \{3,11\}, \{4,9\}, \{4,11\}, \{9,11\}, \{3,4,9\}, \{3,4,11\}, \{3,9,11\}, \{4,9,11\}, \{3,4,9,11\} 
    \item Bi(30) : \{1\}, \{3\}, \{4\}, \{8\}, \{9\}, \{10\}, \{11\}, \{3,4\}, \{8,11\} 
    \item Bi(33) : \{3\}, \{4\}, \{7\}, \{10\}, \{11\}, \{12\}, \{13\}, \{15\}, \{3,4\}, \{3,11\}, \{3,15\}, \{4,11\}, \{4,15\}, \{7,13\}, \{10,12\}, \{11,15\}, \{3,4,11\}, \{3,4,15\}, \{3,11,15\}, \{4,11,15\}, \{3,4,11,15\} 
    \item Bi(39) : \{3\}, \{4\}, \{9\}, \{10\}, \{3,4\}, \{9,10\} 
\end{itemize}

Non-integral:
\begin{itemize}
    \item Bi(10) : \{1,8\}, \{1,9\}, \{3,7\}, \{4,7\}, \{1,8,9\}, \{3,4,7\} 
    \item Bi(14) : \{1,8\}, \{1,9\}, \{3,7\}, \{3,8\}, \{3,9\}, \{4,7\}, \{4,8\}, \{4,9\}, \{3,4,7\}, \{3,4,8\}, \{3,4,9\}, \{3,7,9\}, \{4,7,9\}, \{3,4,7,9\} 
    \item Bi(17) : \{3,8\}, \{3,11\}, \{4,8\}, \{4,11\}, \{8,12\}, \{8,13\}, \{11,12\}, \{11,13\}, \{3,4,8\}, \{3,4,11\}, \{3,8,11\}, \{3,8,12\}, \{3,8,13\}, \{3,11,12\}, \{3,11,13\}, \{4,8,11\}, \{4,8,12\}, \{4,8,13\}, \{4,11,12\}, \{4,11,13\}, \{8,11,12\}, \{8,11,13\}, \{8,12,13\}, \{11,12,13\}, \{3,4,8,11\}, \{3,4,8,12\}, \{3,4,8,13\}, \{3,4,11,12\}, \{3,4,11,13\}, \{3,8,11,12\}, \{3,8,11,13\}, \{3,8,12,13\}, \{3,11,12,13\}, \{4,8,11,12\}, \{4,8,11,13\}, \{4,8,12,13\}, \{4,11,12,13\}, \{8,11,12,13\}, \{3,4,8,11,12\}, \{3,4,8,11,13\}, \{3,4,8,12,13\}, \{3,4,11,12,13\}, \{3,8,11,12,13\}, \{4,8,11,12,13\}, \{3,4,8,11,12,13\} 
    \item Bi(30) : \{1,9\}, \{1,10\}, \{1,11\}, \{3,8\}, \{3,9\}, \{3,10\}, \{3,11\}, \{4,8\}, \{4,9\}, \{4,10\}, \{4,11\}, \{9,11\}, \{10,11\}, \{1,9,11\}, \{1,10,11\}, \{3,4,8\}, \{3,4,9\}, \{3,4,10\}, \{3,4,11\}, \{3,8,11\}, \{3,9,11\}, \{3,10,11\}, \{4,8,11\}, \{4,9,11\}, \{4,10,11\}, \{3,4,8,11\}, \{3,4,9,11\}, \{3,4,10,11\}
    \item Bi(33) : \{3,7\}, \{3,10\}, \{3,12\}, \{3,13\}, \{4,7\}, \{4,10\}, \{4,12\}, \{4,13\}, \{7,11\}, \{7,12\}, \{7,15\}, \{10,11\}, \{10,13\}, \{10,15\}, \{11,12\}, \{11,13\}, \{12,15\}, \{13,15\}, \{3,4,7\}, \{3,4,10\}, \{3,4,12\}, \{3,4,13\}, \{3,7,11\}, \{3,7,12\}, \{3,7,13\}, \{3,7,15\}, \{3,10,11\}, \{3,10,12\}, \{3,10,13\}, \{3,10,15\}, \{3,11,12\}, \{3,11,13\}, \{3,12,15\}, \{3,13,15\}, \{4,7,11\}, \{4,7,12\}, \{4,7,13\}, \{4,7,15\}, \{4,10,11\}, \{4,10,12\}, \{4,10,13\}, \{4,10,15\}, \{4,11,12\}, \{4,11,13\}, \{4,12,15\}, \{4,13,15\}, \{7,11,12\}, \{7,11,13\}, \{7,11,15\}, \{7,12,15\}, \{7,13,15\}, \{10,11,12\}, \{10,11,13\}, \{10,11,15\}, \{10,12,15\}, \{10,13,15\}, \{11,12,15\}, \{11,13,15\}, \{3,4,7,11\}, \{3,4,7,12\}, \{3,4,7,13\}, \{3,4,7,15\}, \{3,4,10,11\}, \{3,4,10,12\}, \{3,4,10,13\}, \{3,4,10,15\}, \{3,4,11,12\}, \{3,4,11,13\}, \{3,4,12,15\}, \{3,4,13,15\}, \{3,7,11,12\}, \{3,7,11,13\}, \{3,7,11,15\}, \{3,7,12,15\}, \{3,7,13,15\}, \{3,10,11,12\}, \{3,10,11,13\}, \{3,10,11,15\}, \{3,10,12,15\}, \{3,10,13,15\}, \{3,11,12,15\}, \{3,11,13,15\}, \{4,7,11,12\}, \{4,7,11,13\}, \{4,7,11,15\}, \{4,7,12,15\}, \{4,7,13,15\}, \{4,10,11,12\}, \{4,10,11,13\}, \{4,10,11,15\}, \{4,10,12,15\}, \{4,10,13,15\}, \{4,11,12,15\}, \{4,11,13,15\}, \{7,11,12,15\}, \{7,11,13,15\}, \{10,11,12,15\}, \{10,11,13,15\}, \{3,4,7,11,12\}, \{3,4,7,11,13\}, \{3,4,7,11,15\}, \{3,4,7,12,15\}, \{3,4,7,13,15\}, \{3,4,10,11,12\}, \{3,4,10,11,13\}, \{3,4,10,11,15\}, \{3,4,10,12,15\}, \{3,4,10,13,15\}, \{3,4,11,12,15\}, \{3,4,11,13,15\}, \{3,7,11,12,15\}, \{3,7,11,13,15\}, \{3,10,11,12,15\}, \{3,10,11,13,15\}, \{4,7,11,12,15\}, \{4,7,11,13,15\}, \{4,10,11,12,15\}, \{4,10,11,13,15\}, \{3,4,7,11,12,15\}, \{3,4,7,11,13,15\}, \{3,4,10,11,12,15\}, \{3,4,10,11,13,15\} 
    \item Bi(39) : \{3,9\}, \{3,10\}, \{4,9\}, \{4,10\}, \{3,4,9\}, \{3,4,10\}, \{3,9,10\}, \{4,9,10\}, \{3,4,9,10\} 
\end{itemize}

\section{A proof of non-integrality for a Bianchi group packing}
\label{nonint bianchi proof example}

To prove non-integrality of extended Bianchi group packings, we applied \lemref{nonintpoly} to solve $g V = 0$, where $V$ is an over-determined inversive coordinate matrix of the packing's cluster and part of its orbit. In the case of non-integrality, $g$ will have a nonlinear relation between its entries, guaranteeing a non-integral packing (see \secref{bianchi packings} for details). Below is an example, which proves non-integrality for $\hat{Bi}(17)$ cluster $\{4,8\}$.

\begin{align*}
\left(
\begin{array}{cccccc}
 g(1,1) & g(1,2) & g(1,3) & g(1,4) & g(1,5) & g(1,6) \\
 g(2,1) & g(2,2) & g(2,3) & g(2,4) & g(2,5) & g(2,6) \\
 g(3,1) & g(3,2) & g(3,3) & g(3,4) & g(3,5) & g(3,6) \\
 g(4,1) & g(4,2) & g(4,3) & g(4,4) & g(4,5) & g(4,6) \\
 g(5,1) & g(5,2) & g(5,3) & g(5,4) & g(5,5) & g(5,6) \\
 g(6,1) & g(6,2) & g(6,3) & g(6,4) & g(6,5) & g(6,6) \\
\end{array}
\right).\left(
\begin{array}{cccc}
 \sqrt{17} & 0 & 0 & 1 \\
 2 \sqrt{34} & \sqrt{34} & \sqrt{\frac{17}{2}} & \frac{11}{\sqrt{2}} \\
 \sqrt{17} & 0 & 0 & -1 \\
 0 & \sqrt{17} & 0 & 1 \\
 5 \sqrt{17} & 4 \sqrt{17} & \sqrt{17} & 18 \\
 39 \sqrt{17} & 16 \sqrt{17} & 0 & 103 \\
\end{array}
\right)= 0
\end{align*}
\begin{align*}
\implies
\left\{
\begin{array}{ccl}
g(1,1)&\to& \frac{3 g(1,2)}{\sqrt{2}}-63 g(1,6)\\
g(1,3)&\to& 24 g(1,6)-\sqrt{2} g(1,2)\\
g(1,4)&\to& \sqrt{2} g(1,2)-16 g(1,6)\\
g(1,5)&\to& -\frac{g(1,2)}{\sqrt{2}}\\
g(2,1)&\to& \frac{3 g(2,2)}{\sqrt{2}}-63 g(2,6)\\
g(2,3)&\to& 24 g(2,6)-\sqrt{2} g(2,2)\\
g(2,4)&\to& \sqrt{2} g(2,2)-16 g(2,6)\\
g(2,5)&\to& -\frac{g(2,2)}{\sqrt{2}} \\
g(3,1)&\to& \frac{3 g(3,2)}{\sqrt{2}}-63 g(3,6) \\
g(3,3)&\to& 24 g(3,6)-\sqrt{2} g(3,2) \\
g(3,4)&\to& \sqrt{2} g(3,2)-16 g(3,6) \\
g(3,5)&\to& -\frac{g(3,2)}{\sqrt{2}} \\
g(4,1)&\to& \frac{3 g(4,2)}{\sqrt{2}}-63 g(4,6) \\
g(4,3)&\to& 24 g(4,6)-\sqrt{2} g(4,2) \\
g(4,4)&\to& \sqrt{2} g(4,2)-16 g(4,6) \\
g(4,5)&\to& -\frac{g(4,2)}{\sqrt{2}} \\
g(5,1)&\to& \frac{3 g(5,2)}{\sqrt{2}}-63 g(5,6) \\
g(5,3)&\to& 24 g(5,6)-\sqrt{2} g(5,2) \\
g(5,4)&\to& \sqrt{2} g(5,2)-16 g(5,6) \\
g(5,5)&\to& -\frac{g(5,2)}{\sqrt{2}} \\
g(6,1)&\to& \frac{3 g(6,2)}{\sqrt{2}}-63 g(6,6) \\
g(6,3)&\to& 24 g(6,6)-\sqrt{2} g(6,2) \\
g(6,4)&\to& \sqrt{2} g(6,2)-16 g(6,6) \\
g(6,5)&\to& -\frac{g(6,2)}{\sqrt{2}}
\end{array}
\right.
\end{align*}

\section{Known nontrivial high-dim. packings: data \& proofs}\label{hdpackingdata}

In this section we present packings for quadratic forms whose Coxeter diagrams as computed from Vinberg's algorithm in \cite{Vin72,Mcl10} do not have clusters, as mentioned in \secref{hd}. 

\subsection{$d=1,n=3$}\label{pack13}

In the configuration obtained by Vinberg's algorithm in \cite{Vin72}, we double about 3 to obtain the following Coxeter diagram

\begin{center}
\begin{tikzpicture}[scale=1.5]
\coordinate (1) at (2,0);
\coordinate (2) at (1,0);
\coordinate (3) at (4,0);
\coordinate (5) at (3,0);
\coordinate (6) at (0,0);

\draw[ultra thick] (2) -- (6);
\draw[ultra thick] (3) -- (5);
\draw[thick] (1) -- (2); 
\draw[thick] (1) -- (5);

{    \filldraw[fill=white] (1) circle (2pt) node[above] {\scriptsize 1};}
{    \filldraw[fill=white] (2) circle (2pt) node[above] {\scriptsize 2};}
{    \filldraw[fill=red] (3) circle (2pt) node[above] {\scriptsize 4};}
{    \filldraw[fill=white] (5) circle (2pt) node[above] {\scriptsize 3.2};}
{    \filldraw[fill=red] (6) circle (2pt) node[above] {\scriptsize 3.4};}

\end{tikzpicture}
\end{center}

\noindent arising from the configuration 

\begin{align}
\begin{array}{r|cccc|l}
& \hat{b} & b & bx & by & \text{also equals:}\\\hline
1 & -\frac{\sqrt{2}}{2} & \frac{\sqrt{2}}{2} & \frac{\sqrt{2}}{2} & 0 & 3.1 \\
2 & 0 & 0 & -\frac{\sqrt{2}}{2} & \frac{\sqrt{2}}{2} \\
4 & \sqrt{2} & 0 & \frac{\sqrt{2}}{2} & \frac{\sqrt{2}}{2}\\
3.2 & 0 & 0 & -\frac{\sqrt{2}}{2} & -\frac{\sqrt{2}}{2} \\
3.4 & \sqrt{2} & 0 & \frac{\sqrt{2}}{2} & -\frac{\sqrt{2}}{2}\\
\end{array}
\end{align}

\noindent which has Gram matrix 

\begin{align}
\left(
\begin{array}{ccccc}
 -1 & 0 & \frac{1}{2} & 0 & 1 \\
 0 & -1 & \frac{1}{2} & 1 & 0 \\
 \frac{1}{2} & \frac{1}{2} & -1 & 0 & 0 \\
 0 & 1 & 0 & -1 & 0 \\
 1 & 0 & 0 & 0 & -1 \\
\end{array}
\right).
\end{align}

Interestingly, this is precisely the same as the Apollonian packing, which is also the packing for $\hat{Bi}(1)$. 

\subsection{$d=3,n=5$}\label{pack35}

In the configuration obtained by Vinberg's algorithm in \cite{Mcl10}, we double about 5 to obtain the following Coxeter diagram

\begin{center}
\begin{tikzpicture}[scale=1.5]
\coordinate (5) at (1,0);
\coordinate (4) at (2,0);
\coordinate (3) at (3,0);
\coordinate (1) at (3.5,.5);
\coordinate (6) at (0,0);
\coordinate (7) at (4.5,-.5);
\coordinate (2) at (3.5,-.5);
\coordinate (8) at (4.5,.5);

\draw[thick, double distance=2.5pt] (6) -- (5);
\draw[thick, double distance=0.3pt] (6) -- (5);
\draw[thick] (5) -- (4);
\draw[thick] (4) -- (3);
\draw[thick] (3) -- (1);
\draw[thick] (3) -- (2);
\draw[ultra thick] (1) -- (8);
\draw[ultra thick] (2) -- (7);

{    \filldraw[fill=white] (1) circle (2pt) node[above] {\scriptsize 5.4};}
{    \filldraw[fill=white] (2) circle (2pt) node[below] {\scriptsize 4};}
{    \filldraw[fill=white] (3) circle (2pt) node[above] {\scriptsize 3};}
{    \filldraw[fill=white] (4) circle (2pt) node[above] {\scriptsize 2};}
{    \filldraw[fill=white] (5) circle (2pt) node[above] {\scriptsize 1};}
{    \filldraw[fill=white] (6) circle (2pt) node[above] {\scriptsize 6};}
{    \filldraw[fill=red] (7) circle (2pt) node[below] {\scriptsize 5.7};}
{    \filldraw[fill=red] (8) circle (2pt) node[above] {\scriptsize 7};}

\end{tikzpicture}
\end{center}

\noindent arising from the configuration 

\begin{align}
\begin{array}{r|cccccc|l}
& \hat{b} & b & bx_1 & bx_2 & bx_3 & bx_4 & \text{also equals:}\\\hline
1 & -\frac{\sqrt{2}}{2} & \frac{\sqrt{2}}{2} & \frac{\sqrt{2}}{2} & 0 & 0 & 0 & 5.1\\
2 & 0 & 0 & -\frac{\sqrt{2}}{2} & \frac{\sqrt{2}}{2} & 0 & 0 & 5.2\\
3 & 0 & 0 & 0 & -\frac{\sqrt{2}}{2} & \frac{\sqrt{2}}{2} & 0 & 5.3\\
4 & 0 & 0 & 0 & 0 & -\frac{\sqrt{2}}{2} & \frac{\sqrt{2}}{2}\\
6 & \frac{\sqrt{2}+\sqrt{6}}{2} & \frac{\sqrt{2}-\sqrt{6}}{2} & 0 & 0 & 0 & 0 & 5.6\\
7 & \frac{\sqrt{2}+\sqrt{6}}{2} & \frac{\sqrt{6}-\sqrt{2}}{2} & \frac{\sqrt{2}}{2} & \frac{\sqrt{2}}{2} & \frac{\sqrt{2}}{2} & \frac{\sqrt{2}}{2}\\
5.4 & 0 & 0 & 0 & 0 & -\frac{\sqrt{2}}{2} & -\frac{\sqrt{2}}{2}\\
5.7 & \frac{\sqrt{2}+\sqrt{6}}{2} & \frac{\sqrt{6}-\sqrt{2}}{2} & \frac{\sqrt{2}}{2} & \frac{\sqrt{2}}{2} & \frac{\sqrt{2}}{2} & -\frac{\sqrt{2}}{2}
\end{array}
\end{align}

\noindent which has Gram matrix 

\begin{align}
\left(
\begin{array}{cccccccc}
 -1 & 0 & \frac{1}{2} & 0 & 0 & 0 & 0 & 1 \\
 0 & -1 & \frac{1}{2} & 0 & 0 & 0 & 1 & 0 \\
 \frac{1}{2} & \frac{1}{2} & -1 & \frac{1}{2} & 0 & 0 & 0 & 0 \\
 0 & 0 & \frac{1}{2} & -1 & \frac{1}{2} & 0 & 0 & 0 \\
 0 & 0 & 0 & \frac{1}{2} & -1 & \frac{\sqrt{3}}{2} & 0 & 0 \\
 0 & 0 & 0 & 0 & \frac{\sqrt{3}}{2} & -1 & 0 & 0 \\
 0 & 1 & 0 & 0 & 0 & 0 & -1 & 0 \\
 1 & 0 & 0 & 0 & 0 & 0 & 0 & -1 \\
\end{array}
\right).
\end{align}

\subsection{$d=3,n=6$}\label{pack36}

\begin{claim}\label{claim:packing36}
The following inversive coordinates generate a subgroup of the group of isometries obtained by Vinberg's algorithm in \cite{Mcl10}.

\begin{align}\label{coords6}
\begin{array}{c|ccccccc|l}
&\hat{b} & b & bx_1 & bx_2 & bx_3 & bx_4 & bx_5 & \text{def'd as}:\\\hline
9& 0& 0& 0& 0& 0& 0& -1 & 3.6\\
10& 0& 0& 0& 0& 0& -\frac{\sqrt{2}}{2} & \frac{\sqrt{2}}{2} & 3.5\\
11& 0& 0& 0& -\frac{\sqrt{2}}{2} & 0& \frac{\sqrt{2}}{2} & 0 & 3.4\\
12& 0& 0& 0& -\frac{\sqrt{2}}{2} & \frac{\sqrt{2}}{2} & 0& 0 & 3\\
13& -\frac{\sqrt{2}}{2} & \frac{\sqrt{2}}{2} & \frac{\sqrt{2}}{2} & 0& 0& 0& 0 & 1\\
14& \sqrt{2} & \sqrt{2} & 0& \frac{\sqrt{6}}{2} & \frac{\sqrt{6}}{2} & 0& 0 & (2.1.2.7).3.(2.1).\nye{7}\\
15& \frac{\sqrt{2}+\sqrt{6}}{2} & \frac{\sqrt{2}-\sqrt{6}}{2} & 0& 0& 0& 0& 0 & 7\\
16& \frac{\sqrt{2}+\sqrt{6}}{2} & \frac{\sqrt{6}-\sqrt{2}}{2} & 0& \sqrt{2} & 0& 0& 0 & 2.7\\
17& \frac{\sqrt{2}+\sqrt{6}}{2} & \frac{\sqrt{6}-\sqrt{2}}{2} & \frac{\sqrt{2}}{2} & \frac{\sqrt{2}}{2} & \frac{\sqrt{2}}{2} & \frac{\sqrt{2}}{2} & 0 & 3.8
\end{array}
\end{align}

\end{claim}

This is \eqref{coords6}'s Gram matrix: 

\begin{align}
\left(
\begin{array}{ccccccccc}
 -1 & \frac{\sqrt{2}}{2} & 0 & 0 & 0 & 0 & 0 & 0 & 0 \\
 \frac{\sqrt{2}}{2} & -1 & \frac{1}{2} & 0 & 0 & 0 & 0 & 0 & \frac{1}{2} \\
 0 & \frac{1}{2} & -1 & -\frac{1}{2} & 0 & \frac{\sqrt{3}}{2} & 0 & 1 & 0 \\
 0 & 0 & -\frac{1}{2} & -1 & 0 & 0 & 0 & 1 & 0 \\
 0 & 0 & 0 & 0 & -1 & 0 & \frac{\sqrt{3}}{2} & \frac{1}{2} & 0 \\
 0 & 0 & \frac{\sqrt{3}}{2} & 0 & 0 & -1 & 1 & 0 & 0 \\
 0 & 0 & 0 & 0 & \frac{\sqrt{3}}{2} & 1 & -1 & 0 & 0 \\
 0 & 0 & 1 & 1 & \frac{1}{2} & 0 & 0 & -1 & 0 \\
 0 & \frac{1}{2} & 0 & 0 & 0 & 0 & 0 & 0 & -1 \\
\end{array}
\right).
\end{align}

\begin{lemma}
\eqref{coords6} has empty interior in $\R^5$, and extends by Poincar\'{e} extension to a hyperbolic polytope of finite volume. 
\end{lemma}
\begin{proof}
We first show that the configuration has empty interior: 
\begin{align}
-x_5&>0&&\tag{3.6.9}\\
\implies x_5&<0\nonumber\\
-\frac{x_4}{\sqrt{2}}+\frac{x_5}{\sqrt{2}}&>0&&\tag{3.6.10}\\
\implies x_4<x_5&<0\nonumber\\
-\frac{x_2}{\sqrt{2}}+\frac{x_4}{\sqrt{2}}&>0&&\tag{3.6.11}\\
\implies x_2<x_4&<0\nonumber\\
\label{ineq368}\left(\frac{\sqrt{6}+\sqrt{2}}{2}\right)^2-x_1^2-\left(\sqrt{3}+1-x_2\right)^2-x_3^2-x_4^2-x_5^2&>0&&\tag{3.6.16}\\
\implies\left(\frac{\sqrt{6}+\sqrt{2}}{2}\right)^2>x_1^2+\left(\sqrt{3}+1-x_2\right)^2+x_3^2+x_4^2+x_5^2&\ge\left(\sqrt{3}+1-x_2\right)^2\nonumber\\
&>\left(\sqrt{3}+1\right)^2\nonumber\\
&=2\left(\frac{\sqrt{6}+\sqrt{2}}{2}\right)^2,\nonumber
\end{align}

\noindent a contradiction. Hence no point $(x_i)_{i=1}^5\in\R^5$ lies in the mutual interior of the specified configuration. 

\eqref{ineq368} gives bounds for each coordinate:
\begin{align*}
&\left(\frac{\sqrt{6}-\sqrt{2}}{2}\right)^2-x_1^2-\left(\sqrt{3}+1-x_2\right)^2-x_3^2-x_4^2-x_5^2>0\\
\implies&\left(\frac{\sqrt{6}-\sqrt{2}}{2}\right)^2>x_1^2+\left(\sqrt{3}+1-x_2\right)^2+x_3^2+x_4^2+x_5^2\ge x_i^2\\
\implies&\abs{x_i}\le\frac{\sqrt{6}-\sqrt{2}}{2}
\end{align*}

\noindent for $1\le i\le5$. Since the intersection of the respective Poincar\'{e} extensions of the circles is bounded and does not meet the boundary of $\H^6$, it must be of finite volume.
\end{proof}

\begin{thm}
\eqref{coords6} generates a sphere packing in $\R^5$ through the cluster $\{12\}$. 
\end{thm}
\begin{proof}
Application of \thmref{structthm} to the following Coxeter diagram proves the result. 
\begin{center}
\begin{tikzpicture}[scale=1.5]%n=6
    \coordinate (one) at (0, 0);
    \coordinate (two) at (1, 0);
    \coordinate (three) at (1, 1);
    \coordinate (four) at (3, 0);
    \coordinate (five) at (4, 1/2);
    \coordinate (six) at (2, 1);
    \coordinate (seven) at (3,1);
    \coordinate (eight) at (2,0);
    \coordinate (nine) at (0,1);
    
    \draw[thick, double distance = 2pt] (one) -- (two);
    \draw[thick] (two) -- (nine);
    \draw[thick] (two) -- (three);
    \draw[ultra thick] (three) -- (eight);
    \draw[thick, double distance = 2.5pt] (three) -- (six);
    \draw[thick, double distance = .3pt] (three) -- (six);
    \draw[ultra thick] (six) -- (seven);
    \draw[thick, double distance = 2.5pt] (five) -- (seven);
    \draw[thick, double distance = .3pt] (five) -- (seven);
    \draw[ultra thick] (four) -- (eight);
    \draw[thick] (five) -- (eight);

{    \filldraw[fill=white] (one) circle (2pt) node[below] {\scriptsize 9};}
{    \filldraw[fill=white] (two) circle (2pt) node[below] {\scriptsize 10};}
{    \filldraw[fill=white] (three) circle (2pt) node[above] {\scriptsize 11};}
{    \filldraw[fill=red] (four) circle (2pt) node[below] {\scriptsize 12};}
{    \filldraw[fill=white] (five) circle (2pt) node[below] {\scriptsize 13};}
{    \filldraw[fill=white] (six) circle (2pt) node[above] {\scriptsize 14};}
{    \filldraw[fill=white] (seven) circle (2pt) node[above] {\scriptsize 15};}
{    \filldraw[fill=white] (eight) circle (2pt) node[below] {\scriptsize 16};}
{    \filldraw[fill=white] (nine) circle (2pt) node[above] {\scriptsize 17};}
\end{tikzpicture}
\end{center}

\end{proof}

\subsection{$d=3,n=7$}\label{pack37}

\begin{claim}\label{claim:packing37}
The following inversive coordinates generate a subgroup of the group of isometries obtained by Vinberg's algorithm in \cite{Mcl10}.

\setcounter{MaxMatrixCols}{20}
\begin{align}\label{coords7}
\begin{array}{c|cccccccc|l}
& \hat{b} & b & bx_1 & bx_2 & bx_3 & bx_4  & bx_5 & bx_6\\\hline
10& 0 & 0 & 0 & 0 & 0 & 0 & 0 & -1 & 3.7\\
11& 0 & 0 & 0 & 0 & 0 & 0 & -\frac{\sqrt{2}}{2} & \frac{\sqrt{2}}{2} & 3.6\\
12& 0 & 0 & 0 & 0 & 0 & -\frac{\sqrt{2}}{2} & \frac{\sqrt{2}}{2} & 0 & 3.5\\
13& 0 & 0 & 0 & -\frac{\sqrt{2}}{2} & 0 & \frac{\sqrt{2}}{2} & 0 & 0 & 3.4\\
14& 0 & 0 & 0 & -\frac{\sqrt{2}}{2} & \frac{\sqrt{2}}{2} & 0 & 0 & 0 & 3\\
15& -\frac{\sqrt{2}}{2} & \frac{\sqrt{2}}{2} & \frac{\sqrt{2}}{2} & 0 & 0 & 0 & 0 & 0 & 1\\
16& \sqrt{2} & \sqrt{2} & 0 & \frac{\sqrt{6}}{2} & \frac{\sqrt{6}}{2} & 0 & 0 & 0 & (2.1.2.8).3.(2.1).\nye{8}\\
17& \frac{\sqrt{2}+\sqrt{6}}{2} & \frac{\sqrt{2}-\sqrt{6}}{2} & 0 & 0 & 0 & 0 & 0 & 0 & 8\\
18& \frac{\sqrt{2}+\sqrt{6}}{2} & \frac{\sqrt{6}-\sqrt{2}}{2} & 0 & \sqrt{2} & 0 & 0 & 0 & 0 & 2.8\\
19& \frac{\sqrt{2}+\sqrt{6}}{2} & \frac{\sqrt{6}-\sqrt{2}}{2} & \frac{\sqrt{2}}{2} & \frac{\sqrt{2}}{2} & \frac{\sqrt{2}}{2} & \frac{\sqrt{2}}{2} & 0 & 0 & 3.9\\
\end{array}
\end{align}

\end{claim}

This is \eqref{coords7}'s Gram matrix: 

\begin{align}
\left(
\begin{array}{cccccccccc}
 -1 & \frac{\sqrt{2}}{2} & 0 & 0 & 0 & 0 & 0 & 0 & 0 & 0 \\
 \frac{\sqrt{2}}{2}  & -1 & \frac{1}{2} & 0 & 0 & 0 & 0 & 0 & 0 & 0 \\
 0 & \frac{1}{2} & -1 & \frac{1}{2} & 0 & 0 & 0 & 0 & 0 & \frac{1}{2} \\
 0 & 0 & \frac{1}{2} & -1 & -\frac{1}{2} & 0 & \frac{\sqrt{3}}{2} & 0 & 1 & 0 \\
 0 & 0 & 0 & -\frac{1}{2} & -1 & 0 & 0 & 0 & 1 & 0 \\
 0 & 0 & 0 & 0 & 0 & -1 & 0 & \frac{\sqrt{3}}{2} & \frac{1}{2} & 0 \\
 0 & 0 & 0 & \frac{\sqrt{3}}{2} & 0 & 0 & -1 & 1 & 0 & 0 \\
 0 & 0 & 0 & 0 & 0 & \frac{\sqrt{3}}{2} & 1 & -1 & 0 & 0 \\
 0 & 0 & 0 & 1 & 1 & \frac{1}{2} & 0 & 0 & -1 & 0 \\
 0 & 0 & \frac{1}{2} & 0 & 0 & 0 & 0 & 0 & 0 & -1 \\
\end{array}
\right).
\end{align}

\begin{lemma}
\eqref{coords7} has empty interior in $\R^6$, and extends by Poincar\'{e} extension to a hyperbolic polytope of finite volume. 
\end{lemma}
\begin{proof}
We first show that the configuration has empty interior: 
\begin{align}
-x_6&>0&&\tag{3.7.10}\\
\implies x_6&<0\nonumber\\
-\frac{x_5}{\sqrt{2}}+\frac{x_6}{\sqrt{2}}&>0&&\tag{3.7.11}\\
\implies x_5<x_6&<0\nonumber\\
-\frac{x_4}{\sqrt{2}}+\frac{x_5}{\sqrt{2}}&>0&&\tag{3.7.12}\\
\implies x_4<x_5&<0\nonumber\\
-\frac{x_2}{\sqrt{2}}+\frac{x_4}{\sqrt{2}}&>0&&\tag{3.7.13}\\
\implies x_2<x_4&<0\nonumber\\
\label{ineq379}\left(\frac{\sqrt{6}+\sqrt{2}}{2}\right)^2-x_1^2-\left(\sqrt{3}+1-x_2\right)^2-x_3^2-x_4^2-x_5^2-x_6^2&>0&&\tag{3.7.18}\\
\implies\left(\frac{\sqrt{6}+\sqrt{2}}{2}\right)^2>x_1^2+\left(\sqrt{3}+1-x_2\right)^2+x_3^2+x_4^2+x_5^2+x_6^2&\ge\left(\sqrt{3}+1-x_2\right)^2\nonumber\\
&>\left(\sqrt{3}+1\right)^2\nonumber\\
&=2\left(\frac{\sqrt{6}+\sqrt{2}}{2}\right)^2,\nonumber
\end{align}

\noindent a contradiction. Hence no point $(x_i)_{i=1}^6\in\R^6$ lies in the mutual interior of the specified configuration. 

\eqref{ineq379} gives bounds for each coordinate:
\begin{align*}
&\left(\frac{\sqrt{6}-\sqrt{2}}{2}\right)^2-x_1^2-\left(\sqrt{3}+1-x_2\right)^2-\sum\limits_{i=3}^6x_i^2>0\\
\implies&\left(\frac{\sqrt{6}-\sqrt{2}}{2}\right)^2>x_1^2+\left(\sqrt{3}+1-x_2\right)^2+\sum\limits_{i=3}^6x_i^2\ge x_i^2\\
\implies&\abs{x_i}\le\frac{\sqrt{6}-\sqrt{2}}{2}
\end{align*}

\noindent for $1\le i\le6,i\neq2$ and $\frac{2\sqrt{3}+2-\sqrt{6}+\sqrt{2}}{2}\le x_2\le\frac{2\sqrt{3}+2+\sqrt{6}-\sqrt{2}}{2}$. Since the intersection of the respective Poincar\'{e} extensions of the circles is bounded and does not meet the boundary of $\H^7$, it must be of finite volume.
\end{proof}

\begin{thm}
\eqref{coords7} generates a sphere packing in $\R^6$ through the cluster $\{14\}$. 
\end{thm}
\begin{proof}
Application of \thmref{structthm} to the following Coxeter diagram proves the result. 

\begin{center}
\begin{tikzpicture}[scale=1.5]%n=7
    \coordinate (one) at (0, 0);
    \coordinate (four) at (2,1);
    \coordinate (seven) at (3,1);
    \coordinate (eight) at (4,1);
    \coordinate (six) at (4, 0);
    \coordinate (two) at (0, 1);
    \coordinate (three) at (1,1);
    \coordinate (ten) at (1,0);
    \coordinate (five) at (2,0);
    \coordinate (nine) at (3,0);
    
    \draw[thick, double distance = 2pt] (one) -- (two);
    \draw[thick] (two) -- (three);
    \draw[thick] (three) -- (four);
    \draw[thick] (three) -- (ten);
    \draw[thick, double distance = 2.5pt] (six) -- (eight);
    \draw[thick, double distance = .3pt] (six) -- (eight);
    \draw[thick, double distance = 2.5pt] (four) -- (seven);
    \draw[thick, double distance = .3pt] (four) -- (seven); 
    \draw[thick] (six) -- (nine);
    \draw[ultra thick] (five) -- (nine);    
    \draw[ultra thick] (seven) -- (eight);
    \draw[thick, double distance=2pt] (four) -- (nine);

{    \filldraw[fill=white] (one) circle (2pt) node[below] {\scriptsize 10};}
{    \filldraw[fill=white] (two) circle (2pt) node[above] {\scriptsize 11};}
{    \filldraw[fill=white] (three) circle (2pt) node[above] {\scriptsize 12};}
{    \filldraw[fill=white] (four) circle (2pt) node[above] {\scriptsize 13};}
{    \filldraw[fill=red] (five) circle (2pt) node[below] {\scriptsize 14};}
{    \filldraw[fill=white] (six) circle (2pt) node[below] {\scriptsize 15};}
{    \filldraw[fill=white] (seven) circle (2pt) node[above] {\scriptsize 16};}
{    \filldraw[fill=white] (eight) circle (2pt) node[above] {\scriptsize 17};}
{    \filldraw[fill=white] (nine) circle (2pt) node[below] {\scriptsize 18};}
{    \filldraw[fill=white] (ten) circle (2pt) node[below] {\scriptsize 19};}
\end{tikzpicture}
\end{center}
\end{proof}

\subsection{$d=3,n=8$}\label{pack38}

\begin{claim}\label{claim:packing38}
The following inversive coordinates generate a subgroup of the group of isometries obtained by Vinberg's algorithm in \cite{Mcl10}.

\begin{align}\label{coords8}
\begin{array}{c|ccccccccc|l}
&\hat{b} & b & bx_1 & bx_2 & bx_3 & bx_4 & bx_5 & bx_6 & bx_7 & \text{def'd as}:\\\hline
11& 0 & 0 & 0 & 0 & 0 & 0 & 0 & 0 & -1 & 3.8\\
12& 0 & 0 & 0 & 0 & 0 & 0 & 0 & -\frac{\sqrt{2}}{2} & \frac{\sqrt{2}}{2} & 3.7\\
13& 0 & 0 & 0 & 0 & 0 & 0 & -\frac{\sqrt{2}}{2} & \frac{\sqrt{2}}{2} & 0 & 3.6\\
14& 0 & 0 & 0 & 0 & 0 & -\frac{\sqrt{2}}{2} & \frac{\sqrt{2}}{2} & 0 & 0 & 3.5\\
15& 0 & 0 & 0 & -\frac{\sqrt{2}}{2} & 0 & \frac{\sqrt{2}}{2} & 0 & 0 & 0 & 3.4\\
16& 0 & 0 & 0 & -\frac{\sqrt{2}}{2} & \frac{\sqrt{2}}{2} & 0 & 0 & 0 & 0 & 3\\
17& -\frac{\sqrt{2}}{2} & \frac{\sqrt{2}}{2} & \frac{\sqrt{2}}{2} & 0 & 0 & 0 & 0 & 0 & 0 & 1\\
18& \sqrt{2} & \sqrt{2} & 0 & \frac{\sqrt{6}}{2} & \frac{\sqrt{6}}{2} & 0 & 0 & 0 & 0 & (2.1.2.9).3.(2.1).\nye{9}\\
19& \frac{\sqrt{2}+\sqrt{6}}{2} & \frac{\sqrt{2}-\sqrt{6}}{2} & 0 & 0 & 0 & 0 & 0 & 0 & 0 & 9\\
20& \frac{\sqrt{2}+\sqrt{6}}{2} & \frac{\sqrt{6}-\sqrt{2}}{2} & 0 & \sqrt{2} & 0 & 0 & 0 & 0 & 0 & 2.9\\
21& \frac{\sqrt{2}+\sqrt{6}}{2} & \frac{\sqrt{6}-\sqrt{2}}{2} & \frac{\sqrt{2}}{2} & \frac{\sqrt{2}}{2} & \frac{\sqrt{2}}{2} & \frac{\sqrt{2}}{2} & 0 & 0 & 0 & 3.10\\
\end{array}
\end{align}

\end{claim}

This is \eqref{coords8}'s Gram matrix: 

\begin{align}
\left(
\begin{array}{ccccccccccc}
 -1 & \frac{1}{\sqrt{2}} & 0 & 0 & 0 & 0 & 0 & 0 & 0 & 0 & 0 \\
 \frac{1}{\sqrt{2}} & -1 & \frac{1}{2} & 0 & 0 & 0 & 0 & 0 & 0 & 0 & 0 \\
 0 & \frac{1}{2} & -1 & \frac{1}{2} & 0 & 0 & 0 & 0 & 0 & 0 & 0 \\
 0 & 0 & \frac{1}{2} & -1 & \frac{1}{2} & 0 & 0 & 0 & 0 & 0 & \frac{1}{2} \\
 0 & 0 & 0 & \frac{1}{2} & -1 & -\frac{1}{2} & 0 & \frac{\sqrt{3}}{2} & 0 & 1 & 0 \\
 0 & 0 & 0 & 0 & -\frac{1}{2} & -1 & 0 & 0 & 0 & 1 & 0 \\
 0 & 0 & 0 & 0 & 0 & 0 & -1 & 0 & \frac{\sqrt{3}}{2} & \frac{1}{2} & 0 \\
 0 & 0 & 0 & 0 & \frac{\sqrt{3}}{2} & 0 & 0 & -1 & 1 & 0 & 0 \\
 0 & 0 & 0 & 0 & 0 & 0 & \frac{\sqrt{3}}{2} & 1 & -1 & 0 & 0 \\
 0 & 0 & 0 & 0 & 1 & 1 & \frac{1}{2} & 0 & 0 & -1 & 0 \\
 0 & 0 & 0 & \frac{1}{2} & 0 & 0 & 0 & 0 & 0 & 0 & -1 \\
\end{array}
\right).
\end{align}

\begin{lemma}
\eqref{coords8} has empty interior in $\R^7$, and extends by Poincar\'{e} extension to a hyperbolic polytope of finite volume. 
\end{lemma}
\begin{proof}
We first show that the configuration has empty interior: 
\begin{align}
-x_7&>0&&\tag{3.8.11}\\
\implies x_7&<0\nonumber\\
-\frac{x_6}{\sqrt{2}}+\frac{x_7}{\sqrt{2}}&>0&&\tag{3.8.12}\\
\implies x_6<x_7&<0\nonumber\\
-\frac{x_5}{\sqrt{2}}+\frac{x_6}{\sqrt{2}}&>0&&\tag{3.8.13}\\
\implies x_5<x_6&<0\nonumber\\
-\frac{x_4}{\sqrt{2}}+\frac{x_5}{\sqrt{2}}&>0&&\tag{3.8.14}\\
\implies x_4<x_5&<0\nonumber\\
-\frac{x_2}{\sqrt{2}}+\frac{x_4}{\sqrt{2}}&>0&&\tag{3.8.15}\\
\implies x_2<x_4&<0\nonumber\\
\label{ineq3810}\left(\frac{\sqrt{6}+\sqrt{2}}{2}\right)^2-x_1^2-\left(\sqrt{3}+1-x_2\right)^2-\sum\limits_{i=3}^7x_i^2&>0&&\tag{3.8.20}\\
\implies\left(\frac{\sqrt{6}+\sqrt{2}}{2}\right)^2>x_1^2+\left(\sqrt{3}+1-x_2\right)^2+\sum\limits_{i=3}^7x_i^2&\ge\left(\sqrt{3}+1-x_2\right)^2\nonumber\\
&>\left(\sqrt{3}+1\right)^2\nonumber\\
&=2\left(\frac{\sqrt{6}+\sqrt{2}}{2}\right)^2,\nonumber
\end{align}

\noindent a contradiction. Hence no point $(x_i)_{i=1}^7\in\R^7$ lies in the mutual interior of the specified configuration. 

\eqref{ineq3810} gives bounds for each coordinate:
\begin{align*}
\left(\frac{\sqrt{6}-\sqrt{2}}{2}\right)^2-x_1^2-\left(\sqrt{3}+1-x_2\right)^2-\sum\limits_{i=3}^7x_i^2&>0\\
\implies\left(\frac{\sqrt{6}-\sqrt{2}}{2}\right)^2>x_1^2+\left(\sqrt{3}+1-x_2\right)^2+\sum\limits_{i=3}^7x_i^2&\ge x_i^2\\
\implies\abs{x_i}&\le\frac{\sqrt{6}-\sqrt{2}}{2}
\end{align*}

\noindent for $1\le i\le7,i\neq2$ and $\frac{2\sqrt{3}+2-\sqrt{6}+\sqrt{2}}{2}\le x_2\le\frac{2\sqrt{3}+2+\sqrt{6}-\sqrt{2}}{2}$. Since the intersection of the respective Poincar\'{e} extensions of the circles is bounded and does not meet the boundary of $\H^7$, it must be of finite volume.
\end{proof}

\begin{thm}
\eqref{coords8} generates a sphere packing in $\R^7$ through the cluster $\{16\}$. 
\end{thm}
\begin{proof}
Application of \thmref{structthm} to the following Coxeter diagram proves the result. 

\begin{center}
\begin{tikzpicture}[scale=1.5]%n=8
    \coordinate (six) at (0, 0);
    \coordinate (ten) at (1, 0);
    \coordinate (five) at (2, 0);
    \coordinate (four) at (3, 0);
    \coordinate (three) at (4, 0);
    \coordinate (two) at (5,0);
    \coordinate (one) at (6,0);
    \coordinate (seven) at (0, 1);
    \coordinate (nine) at (1,1);
    \coordinate (eight) at (2,1);
    \coordinate (eleven) at (3,1);

    \draw[thick, double distance=2.5pt] (five) -- (eight);
    \draw[thick, double distance=0.3pt] (five) -- (eight);
    \draw[thick, double distance=2.5pt] (seven) -- (nine);
    \draw[thick, double distance=0.3pt] (seven) -- (nine);
    \draw[ultra thick] (eight) -- (nine);
    \draw[ultra thick] (six) -- (ten);
    \draw[ultra thick] (ten) -- (five);
    \draw[thick] (four) -- (eleven);
    \draw[thick] (four) -- (five);
    \draw[thick, double distance=2pt] (one) -- (two);
    \draw[thick] (two) -- (three);
    \draw[thick] (three) -- (four);
    \draw[thick] (seven) -- (ten);
    
{    \filldraw[fill=white] (one) circle (2pt) node[below] {\scriptsize 11};}
{    \filldraw[fill=white] (two) circle (2pt) node[below] {\scriptsize 12};}
{    \filldraw[fill=white] (three) circle (2pt) node[below] {\scriptsize 13};}
{    \filldraw[fill=white] (four) circle (2pt) node[below] {\scriptsize 14};}
{    \filldraw[fill=white] (five) circle (2pt) node[below] {\scriptsize 15};}
{    \filldraw[fill=red] (six) circle (2pt) node[below] {\scriptsize 16};}
{    \filldraw[fill=white] (seven) circle (2pt) node[above] {\scriptsize 17};}
{    \filldraw[fill=white] (eight) circle (2pt) node[above] {\scriptsize 18};}
{    \filldraw[fill=white] (nine) circle (2pt) node[above] {\scriptsize 19};}
{    \filldraw[fill=white] (ten) circle (2pt) node[below] {\scriptsize 20};}
{    \filldraw[fill=white] (eleven) circle (2pt) node[above] {\scriptsize 21};}
\end{tikzpicture}
\end{center}
\end{proof}

\subsection{$d=3,n=10$}\label{pack310}

\begin{claim}\label{claim:packing310}
The following inversive coordinates generate a subgroup of the group of isometries obtained by Vinberg's algorithm in \cite{Mcl10}.

\setcounter{MaxMatrixCols}{20}
\begin{align}\label{coords10}
\begin {array}{c|ccccccccccc|l}
&\hat{b} & b & bx_1 & bx_2 & bx_3 & bx_4 & bx_5 & bx_6 & bx_7 & bx_8 & bx_9 & \text{def'd as}:\\\hline
15& 0 & 0 & 0 & 0 & 0 & 0 & 0 & 0 & 0 & 0 & - 1 & 3.10\\
16& 0 & 0 & 0 & 0 & 0 & 0 & 0 & 0 & 0 & - \frac {\sqrt {2}} {2} & \frac{\sqrt {2}} {2} & 3.9\\
17& 0 & 0 & 0 & 0 & 0 & 0 & 0 & 0 & - \frac {\sqrt {2}} {2} & \frac{\sqrt {2}} {2} & 0 & 3.8\\
18& 0 & 0 & 0 & 0 & 0 & 0 & 0 & - \frac {\sqrt {2}} {2} & \frac {\sqrt{2}} {2} & 0 & 0 & 3.7\\
19& 0 & 0 & 0 & 0 & 0 & 0 & - \frac {\sqrt {2}} {2} & \frac {\sqrt {2}}{2} & 0 & 0 & 0 & 3.6\\
20& 0 & 0 & 0 & 0 & 0 & - \frac {\sqrt {2}} {2} &\frac {\sqrt {2}} {2} & 0 & 0 & 0 & 0 & 3.5\\
21& 0 & 0 & 0 & - \frac {\sqrt {2}} {2} & 0 & \frac {\sqrt {2}} {2} & 0 & 0 & 0 & 0 & 0 & 3.4\\
22& 0 & 0 & 0 & - \frac {\sqrt {2}} {2} & \frac {\sqrt {2}} {2} & 0 & 0 & 0 & 0 & 0 & 0 & 3\\
23& -\frac {\sqrt {2}} {2} & \frac {\sqrt {2}} {2} & \frac {\sqrt {2}}{2} & 0 & 0 & 0 & 0 & 0 & 0 & 0 & 0 & 1\\
24& \sqrt {2} & \sqrt {2} & 0 & \frac{\sqrt{6}}{2} & \frac{\sqrt{6}}{2} & 0 & 0 & 0 & 0 & 0 & 0 & \footnotemark\\
25& 2 \left (1 + \sqrt {3} \right) & 2 \left (-1 + \sqrt {3} \right) & 1 & 1 & 1 & 1 & 1 & 1 & 1 & 1 & 1 & 3.14\\
26& \frac {\sqrt{2} + \sqrt {6}} {2} & \frac {\sqrt{2} - \sqrt {6}} {2} & 0 & 0 & 0 & 0 & 0 & 0 & 0 & 0 & 0 & 11\\
27& \frac {\sqrt{2} + \sqrt {6}} {2} & \frac {\sqrt{6} - \sqrt {6}} {2} & 0 & \sqrt {2} & 0 & 0 & 0 & 0 & 0 & 0 & 0 & 2.11\\
28& \frac {\sqrt{2} + \sqrt {6}} {2} & \frac {\sqrt{6} - \sqrt {6}} {2} & \frac {\sqrt {2}} {2} & \frac {\sqrt {2}} {2} & \frac {\sqrt{2}} {2} & \frac {\sqrt {2}} {2} & 0 & 0 & 0 & 0 & 0 & 3.12\\
29& \frac {5\sqrt{2} + \sqrt {6}} {2} & \frac {5\sqrt{2} - \sqrt {6}} {2} & \frac{\sqrt{6}}{2} & \frac{\sqrt{6}}{2} & \frac{\sqrt{6}}{2} & \frac{\sqrt{6}}{2} & \frac{\sqrt{6}}{2} & \frac{\sqrt{6}}{2} & \frac{\sqrt{6}}{2} & \frac{\sqrt{6}}{2} & 0 & 3.13
\end {array}
\end{align}
\footnotetext{$(2.1.2.11).3.(2.1).\nye{11}$}
\end{claim}

This is \eqref{coords10}'s Gram matrix: 

\begin{align}
\left(
\begin{array}{ccccccccccccccc}
 -1 & \frac{\sqrt{2}}{2} & 0 & 0 & 0 & 0 & 0 & 0 & 0 & 0 & 1 & 0 & 0 & 0 & 0 \\
 \frac{\sqrt{2}}{2} & -1 & \frac{1}{2} & 0 & 0 & 0 & 0 & 0 & 0 & 0 & 0 & 0 & 0 & 0 &
   \frac{\sqrt{3}}{2} \\
 0 & \frac{1}{2} & -1 & \frac{1}{2} & 0 & 0 & 0 & 0 & 0 & 0 & 0 & 0 & 0 & 0 & 0 \\
 0 & 0 & \frac{1}{2} & -1 & \frac{1}{2} & 0 & 0 & 0 & 0 & 0 & 0 & 0 & 0 & 0 & 0 \\
 0 & 0 & 0 & \frac{1}{2} & -1 & \frac{1}{2} & 0 & 0 & 0 & 0 & 0 & 0 & 0 & 0 & 0 \\
 0 & 0 & 0 & 0 & \frac{1}{2} & -1 & \frac{1}{2} & 0 & 0 & 0 & 0 & 0 & 0 & \frac{1}{2} & 0 \\
 0 & 0 & 0 & 0 & 0 & \frac{1}{2} & -1 & -\frac{1}{2} & 0 & \frac{\sqrt{3}}{2} & 0 & 0 & 1 & 0 & 0
   \\
 0 & 0 & 0 & 0 & 0 & 0 & -\frac{1}{2} & -1 & 0 & 0 & 0 & 0 & 1 & 0 & 0 \\
 0 & 0 & 0 & 0 & 0 & 0 & 0 & 0 & -1 & 0 & \frac{\sqrt{2}}{2} & \frac{\sqrt{3}}{2} & \frac{1}{2} &
   0 & 0 \\
 0 & 0 & 0 & 0 & 0 & 0 & \frac{\sqrt{3}}{2} & 0 & 0 & -1 & \sqrt{6} & 1 & 0 & 0 & 2 \\
 1 & 0 & 0 & 0 & 0 & 0 & 0 & 0 & \frac{\sqrt{2}}{2} & \sqrt{6} & -1 & 0 & \sqrt{2} & 0 & 0 \\
 0 & 0 & 0 & 0 & 0 & 0 & 0 & 0 & \frac{\sqrt{3}}{2} & 1 & 0 & -1 & 0 & 0 & 1 \\
 0 & 0 & 0 & 0 & 0 & 0 & 1 & 1 & \frac{1}{2} & 0 & \sqrt{2} & 0 & -1 & 0 & \sqrt{3} \\
 0 & 0 & 0 & 0 & 0 & \frac{1}{2} & 0 & 0 & 0 & 0 & 0 & 0 & 0 & -1 & 0 \\
 0 & \frac{\sqrt{3}}{2} & 0 & 0 & 0 & 0 & 0 & 0 & 0 & 2 & 0 & 1 & \sqrt{3} & 0 & -1 \\
\end{array}
\right).
\end{align}

\begin{lemma}
\eqref{coords10} has empty interior in $\R^9$, and extends by Poincar\'{e} extension to a hyperbolic polytope of finite volume. 
\end{lemma}
\begin{proof}
We first show that the configuration has empty interior: 
\begin{align}
-x_9&>0&&\tag{3.10.15}\\
\implies x_9&<0\nonumber\\
\label{ineq31016}\left(\frac{\sqrt{6}+\sqrt{2}}{4}\right)^2-\sum\limits_{i=1}^{10}\left(\frac{\sqrt{6}+\sqrt{2}}{4}-x_i\right)^2&>0&&\tag{3.10.25}\\
\left(\frac{\sqrt{6}+\sqrt{2}}{4}\right)^2>\sum\limits_{i=1}^{10}\left(\frac{\sqrt{6}+\sqrt{2}}{4}-x_i\right)^2&\ge\left(\frac{\sqrt{6}+\sqrt{2}}{4}-x_9\right)^2,\nonumber\\
&>\left(\frac{\sqrt{6}+\sqrt{2}}{4}\right)^2,\nonumber\\
\end{align}

\noindent a contradiction. Hence no point $(x_i)_{i=1}^9\in\R^9$ lies in the mutual interior of the specified configuration. 

\eqref{ineq31016} gives bounds for each coordinate:
\begin{align*}
&\left(\frac{\sqrt{6}+\sqrt{2}}{4}\right)^2-\sum\limits_{i=1}^{10}\left(\frac{\sqrt{6}+\sqrt{2}}{4}-x_i\right)^2>0\\
\implies&\left(\frac{\sqrt{6}+\sqrt{2}}{4}\right)^2>\sum\limits_{i=1}^{10}\left(\frac{\sqrt{6}+\sqrt{2}}{4}-x_i\right)^2\ge\left(\frac{\sqrt{6}+\sqrt{2}}{4}-x_i\right)^2\\
\implies&0<x_i<\frac{\sqrt{6}+\sqrt{2}}{2}
\end{align*}

\noindent for $1\le i\le9$. Since the intersection of the respective Poincar\'{e} extensions of the circles is bounded and does not meet the boundary of $\H^{10}$, it must be of finite volume.
\end{proof}

\begin{thm}
\eqref{coords10} generates a sphere packing in $\R^9$ through the cluster $\{22\}$. 
\end{thm}
\begin{proof}
Application of \thmref{structthm} to the following Coxeter diagram proves the result. 

\begin{center}
\begin{tikzpicture}[scale=1.5]%n=10
    \coordinate (14) at (1,1.5);
    \coordinate (5) at (0,0);
    \coordinate (6) at (1,0);
    \coordinate (7) at (2,0);
    \coordinate (13) at (4,0);
    \coordinate (8) at (5,0);
    \coordinate (10) at (2,1);
    \coordinate (12) at (3,1.5);
    \coordinate (9) at (4.5,1.5);
    \coordinate (15) at (2,2);
    \coordinate (4) at (0,3);
    \coordinate (3) at (1,3);
    \coordinate (2) at (2,3);
    \coordinate (1) at (3,3);
    \coordinate (11) at (4,3);
    
    \draw[thick, double distance=2pt] (1) -- (2);
    \draw[ultra thick] (1) -- (11);
    \draw[thick] (2) -- (3);
    \draw[thick, double distance=2.5pt] (2) -- (15);
    \draw[thick, double distance=0.3pt] (2) -- (15);
    \draw[thick] (3) -- (4);
    \draw[thick] (4) -- (5);
    \draw[thick] (5) -- (6);
    \draw[thick] (6) -- (7);
    \draw[thick] (6) -- (14);
    \draw[thick, double distance=2.5pt] (7) -- (10);
    \draw[thick, double distance=0.3pt] (7)-- (10);
    \draw[ultra thick] (7) -- (13);
    \draw[ultra thick] (8) -- (13);
    \draw[thick, double distance=2] (9) -- (11);
    \draw[thick, double distance=2.5pt] (9) -- (12);
    \draw[thick, double distance=0.3pt] (9)-- (12);    
    \draw[thick] (9) -- (13);
    \draw[dashed] (10) -- (11);
    \draw[ultra thick] (10) -- (12);
    \draw[dashed] (10) -- (15);
    \draw[dashed] (11) -- (13);
    \draw[ultra thick] (12) -- (15);
    \draw[dashed] (13) -- (15);
    
{    \filldraw[fill=white] (1) circle (2pt) node[above] {\scriptsize 15};}
{    \filldraw[fill=white] (2) circle (2pt) node[above] {\scriptsize 16};}
{    \filldraw[fill=white] (3) circle (2pt) node[above] {\scriptsize 17};}
{    \filldraw[fill=white] (4) circle (2pt) node[above] {\scriptsize 18};}
{    \filldraw[fill=white] (5) circle (2pt) node[below] {\scriptsize 19};}
{    \filldraw[fill=white] (6) circle (2pt) node[below] {\scriptsize 20};}
{    \filldraw[fill=white] (7) circle (2pt) node[below] {\scriptsize 21};}
{    \filldraw[fill=red] (8) circle (2pt) node[below] {\scriptsize 22};}
{    \filldraw[fill=white] (9) circle (2pt) node[right] {\scriptsize 23};}
{    \filldraw[fill=white] (10) circle (2pt) node[left] {\scriptsize 24};}
{    \filldraw[fill=white] (11) circle (2pt) node[above] {\scriptsize 25};}
{    \filldraw[fill=white] (12) circle (2pt) node[above] {\scriptsize 26};}
{    \filldraw[fill=white] (13) circle (2pt) node[below] {\scriptsize 27};}
{    \filldraw[fill=white] (14) circle (2pt) node[above] {\scriptsize 28};}
{    \filldraw[fill=white] (15) circle (2pt) node[left] {\scriptsize 29};}
\end{tikzpicture}
\end{center}
\end{proof}

\subsection{$d=3,n=11$}\label{pack311}

\begin{claim}\label{claim:packing311}
The following inversive coordinates generate a subgroup of the group of isometries obtained by Vinberg's algorithm in \cite{Mcl10}.

\setcounter{MaxMatrixCols}{20}
\begin{align}\label{coords11}
\begin{array}{c|cccccccccccc|l}
& \hat{b} & b & bx_1 & bx_2 & bx_3 & bx_4 & bx_5 & bx_6 & bx_7 & bx_8 & bx_9 & bx_{10} & \text{def'd as}:\\\hline
16& 0& 0& 0& 0& 0& 0& 0& 0& 0& 0& 0& -1 & 2.3.2.11\\
17& 0& 0& 0& 0& 0& 0& 0& 0& 0& 0& -\frac{\sqrt{2}}{2} & \frac{\sqrt{2}}{2} & 2.3.2.10\\
18& 0& 0& 0& 0& 0& 0& 0& 0& 0& -\frac{\sqrt{2}}{2} & \frac{\sqrt{2}}{2} & 0 & 2.3.2.9\\
19& 0& 0& 0& 0& 0& 0& 0& 0& -\frac{\sqrt{2}}{2} & \frac{\sqrt{2}}{2} & 0& 0 & 2.3.2.8\\
20& 0& 0& 0& 0& 0& 0& 0& -\frac{\sqrt{2}}{2} & \frac{\sqrt{2}}{2} & 0& 0& 0 & 2.3.2.7\\
21& 0& 0& 0& 0& 0& 0& -\frac{\sqrt{2}}{2} & \frac{\sqrt{2}}{2} & 0& 0& 0& 0 & 2.3.2.6\\
22& 0& 0& 0& 0& 0& -\frac{\sqrt{2}}{2} & \frac{\sqrt{2}}{2} & 0& 0& 0& 0& 0 & 2.3.2.5\\
23& 0& 0& 0& -\frac{\sqrt{2}}{2} & \frac{\sqrt{2}}{2} & 0& 0& 0& 0& 0& 0& 0 & 3\\
24& 0& 0& -\frac{\sqrt{2}}{2} & 0& 0& \frac{\sqrt{2}}{2} & 0& 0& 0& 0& 0& 0 & 2.3.2.4\\
25& -\frac{\sqrt{2}}{2} & \frac{\sqrt{2}}{2} & \frac{\sqrt{2}}{2} & 0& 0& 0& 0& 0& 0& 0& 0& 0 & 1\\
26& \sqrt{2} & \sqrt{2} & 0& \frac{\sqrt{6}}{2} & \frac{\sqrt{6}}{2} & 0& 0& 0& 0& 0& 0& 0 & \footnotemark\\
27& \frac{\sqrt{2}+\sqrt{6}}{2} & \frac{\sqrt{2}-\sqrt{6}}{2} & 0& 0& 0& 0& 0& 0& 0& 0& 0& 0 & 12\\
28& \frac{\sqrt{2}+\sqrt{6}}{2} & \frac{\sqrt{6}-\sqrt{2}}{2} & 0& \sqrt{2} & 0& 0& 0& 0& 0& 0& 0& 0 & 2.12\\
29& \frac{\sqrt{2}+\sqrt{6}}{2} & \frac{\sqrt{6}-\sqrt{2}}{2} & \frac{\sqrt{2}}{2} & \frac{\sqrt{2}}{2} & \frac{\sqrt{2}}{2} & \frac{\sqrt{2}}{2} & 0& 0& 0& 0& 0& 0 & 2.3.2.13\\
30& \frac{5\sqrt{2}+\sqrt{6}}{2} & \frac{5\sqrt{2}-\sqrt{6}}{2} & \frac{\sqrt{6}}{2} & \frac{\sqrt{6}}{2} & \frac{\sqrt{6}}{2} & \frac{\sqrt{6}}{2} & \frac{\sqrt{6}}{2} & \frac{\sqrt{6}}{2} & \frac{\sqrt{6}}{2} & \frac{\sqrt{6}}{2} & 0& 0 & 2.3.2.14\\
31& \sqrt{6}+\sqrt{2} & \sqrt{6}-\sqrt{2} & \frac{\sqrt{2}}{2} & \frac{\sqrt{2}}{2} & \frac{\sqrt{2}}{2} & \frac{\sqrt{2}}{2} & \frac{\sqrt{2}}{2} & \frac{\sqrt{2}}{2} & \frac{\sqrt{2}}{2} & \frac{\sqrt{2}}{2} & \frac{\sqrt{2}}{2} & \frac{\sqrt{2}}{2} & 2.3.2.15
\end{array}
\end{align}
\footnotetext{$(2.1.2.12).3.(2.1).\nye{12}$}

\end{claim}

This is \eqref{coords11}'s Gram matrix: 

\begin{align}
\left(
\begin{array}{cccccccccccccccc}
 -1 & \frac{\sqrt{2}}{2} & 0 & 0 & 0 & 0 & 0 & 0 & 0 & 0 & 0 & 0 & 0 & 0 & 0 & \frac{\sqrt{2}}{2}
   \\
 \frac{\sqrt{2}}{2} & -1 & \frac{1}{2} & 0 & 0 & 0 & 0 & 0 & 0 & 0 & 0 & 0 & 0 & 0 & 0 & 0 \\
 0 & \frac{1}{2} & -1 & \frac{1}{2} & 0 & 0 & 0 & 0 & 0 & 0 & 0 & 0 & 0 & 0 & \frac{\sqrt{3}}{2} &
   0 \\
 0 & 0 & \frac{1}{2} & -1 & \frac{1}{2} & 0 & 0 & 0 & 0 & 0 & 0 & 0 & 0 & 0 & 0 & 0 \\
 0 & 0 & 0 & \frac{1}{2} & -1 & \frac{1}{2} & 0 & 0 & 0 & 0 & 0 & 0 & 0 & 0 & 0 & 0 \\
 0 & 0 & 0 & 0 & \frac{1}{2} & -1 & \frac{1}{2} & 0 & 0 & 0 & 0 & 0 & 0 & 0 & 0 & 0 \\
 0 & 0 & 0 & 0 & 0 & \frac{1}{2} & -1 & 0 & \frac{1}{2} & 0 & 0 & 0 & 0 & \frac{1}{2} & 0 & 0 \\
 0 & 0 & 0 & 0 & 0 & 0 & 0 & -1 & 0 & 0 & 0 & 0 & 1 & 0 & 0 & 0 \\
 0 & 0 & 0 & 0 & 0 & 0 & \frac{1}{2} & 0 & -1 & \frac{1}{2} & 0 & 0 & 0 & 0 & 0 & 0 \\
 0 & 0 & 0 & 0 & 0 & 0 & 0 & 0 & \frac{1}{2} & -1 & 0 & \frac{\sqrt{3}}{2} & \frac{1}{2} & 0 & 0 &
   \frac{1}{2} \\
 0 & 0 & 0 & 0 & 0 & 0 & 0 & 0 & 0 & 0 & -1 & 1 & 0 & 0 & 2 & \sqrt{3} \\
 0 & 0 & 0 & 0 & 0 & 0 & 0 & 0 & 0 & \frac{\sqrt{3}}{2} & 1 & -1 & 0 & 0 & 1 & 0 \\
 0 & 0 & 0 & 0 & 0 & 0 & 0 & 1 & 0 & \frac{1}{2} & 0 & 0 & -1 & 0 & \sqrt{3} & 1 \\
 0 & 0 & 0 & 0 & 0 & 0 & \frac{1}{2} & 0 & 0 & 0 & 0 & 0 & 0 & -1 & 0 & 0 \\
 0 & 0 & \frac{\sqrt{3}}{2} & 0 & 0 & 0 & 0 & 0 & 0 & 0 & 2 & 1 & \sqrt{3} & 0 & -1 & 0 \\
 \frac{\sqrt{2}}{2} & 0 & 0 & 0 & 0 & 0 & 0 & 0 & 0 & \frac{1}{2} & \sqrt{3} & 0 & 1 & 0 & 0 & -1
   \\
\end{array}
\right).
\end{align}

\begin{lemma}
\eqref{coords11} has empty interior in $\R^{10}$, and extends by Poincar\'{e} extension to a hyperbolic polytope of finite volume. 
\end{lemma}
\begin{proof}
We first show that the configuration has empty interior: 
\begin{align}
-x_{10}&>0&&\tag{3.11.16}\\
\implies x_{10}&<0\nonumber\\
-\frac{x_9}{\sqrt{2}}+\frac{x_{10}}{\sqrt{2}}&>0&&\tag{3.11.17}\\
\implies x_9<x_{10}&<0\nonumber\\
\label{ineq31116}\left(\frac{\sqrt{6}+\sqrt{2}}{4}\right)^2-\sum\limits_{i=1}^{10}\left(\frac{\sqrt{3}+1}{4}-x_i\right)^2&>0&&\tag{3.11.31}\\
\left(\frac{\sqrt{6}+\sqrt{2}}{4}\right)^2>\sum\limits_{i=1}^{10}\left(\frac{\sqrt{3}+1}{4}-x_i\right)^2&\ge\left(\frac{\sqrt{3}+1}{4}-x_9\right)^2+\left(\frac{\sqrt{3}+1}{4}-x_{10}\right)^2,\nonumber\\
&>2\left(\frac{\sqrt{3}+1}{4}\right)^2,\nonumber\\
&=\left(\frac{\sqrt{6}+\sqrt{2}}{4}\right)^2,\nonumber
\end{align}

\noindent a contradiction. Hence no point $(x_i)_{i=1}^{10}\in\R^{10}$ lies in the mutual interior of the specified configuration. 

\eqref{ineq31116} gives bounds for each coordinate:
\begin{align*}
&\left(\frac{\sqrt{6}+\sqrt{2}}{4}\right)^2-\sum\limits_{i=1}^{10}\left(\frac{\sqrt{3}+1}{4}-x_i\right)^2>0\\
\implies&\left(\frac{\sqrt{6}+\sqrt{2}}{4}\right)^2>\sum\limits_{i=1}^{10}\left(\frac{\sqrt{3}+1}{4}-x_i\right)^2\ge\left(\frac{\sqrt{3}+1}{4}-x_i\right)^2\\
\implies&\left(1-\sqrt{2}\right)\frac{\sqrt{3}+1}{4}<x_i<\left(1+\sqrt{2}\right)\frac{\sqrt{3}+1}{4}
\end{align*}

\noindent for $1\le i\le10$. Since the intersection of the respective Poincar\'{e} extensions of the circles is bounded and does not meet the boundary of $\H^{11}$, it must be of finite volume.
\end{proof}

\begin{thm}
\eqref{coords7} generates a sphere packing in $\R^{10}$ through either of the clusters $\{23\}$ or $\{26\}$. 
\end{thm}
\begin{proof}
Application of \thmref{structthm} to the following Coxeter diagram proves the result. 

\begin{center}
\begin{tikzpicture}[scale=1.5]%n=11
    \coordinate (fifteen) at (2.5,2);
    \coordinate (four) at (1.5,4);
    \coordinate (three) at (2,3);
    \coordinate (eleven) at (1,3);
    \coordinate (fourteen) at (0.5,1);
    \coordinate (ten) at (1,1);
    \coordinate (five) at (0.5,4);
    \coordinate (six) at (0,3);
    \coordinate (seven) at (0,2);
    \coordinate (twelve) at (1,2);
    \coordinate (thirteen) at (2,0);
    \coordinate (eight) at (1,0);
    \coordinate (two) at (3,3);
    \coordinate (nine) at (0,0);
    \coordinate (one) at (3,2);
    \coordinate (sixteen) at (2,1);
    
    \draw[thick,double distance=2pt] (one) -- (two);
    \draw[thick,double distance=2pt] (one) -- (sixteen);
    \draw[thick] (two) -- (three);
    \draw[thick] (three) -- (four);
    \draw[thick,double distance=2.5pt] (three) -- (fifteen);
    \draw[thick,double distance=0.3pt] (three) -- (fifteen);
    \draw[thick] (four) -- (five);
    \draw[thick] (five) -- (six);
    \draw[thick] (six) -- (seven);
    \draw[thick] (seven) -- (nine);
    \draw[thick] (seven) -- (fourteen);
    \draw[ultra thick] (eight) -- (thirteen);
    \draw[thick] (nine) -- (ten);
    \draw[thick,double distance=2.5pt] (ten) -- (twelve);
    \draw[thick,double distance=0.3pt] (ten) -- (twelve);
    \draw[thick] (ten) -- (thirteen);
    \draw[thick] (ten) -- (sixteen);
    \draw[ultra thick] (eleven) -- (twelve);
    \draw[dashed] (eleven) -- (fifteen);
    \draw[dashed] (eleven) -- (sixteen);
    \draw[ultra thick] (twelve) -- (fifteen);
    \draw[dashed] (thirteen) -- (fifteen);
    \draw[ultra thick] (thirteen) -- (sixteen);
    
{    \filldraw[fill=white] (one) circle (2pt) node[below] {\scriptsize 16};}
{    \filldraw[fill=white] (two) circle (2pt) node[above] {\scriptsize 17};}
{    \filldraw[fill=white] (three) circle (2pt) node[above] {\scriptsize 18};}
{    \filldraw[fill=white] (four) circle (2pt) node[above] {\scriptsize 19};}
{    \filldraw[fill=white] (five) circle (2pt) node[above] {\scriptsize 20};}
{    \filldraw[fill=white] (six) circle (2pt) node[left] {\scriptsize 21};}
{    \filldraw[fill=white] (seven) circle (2pt) node[left] {\scriptsize 22};}
{    \filldraw[fill=red] (eight) circle (2pt) node[below] {\scriptsize 23};}
{    \filldraw[fill=white] (nine) circle (2pt) node[below] {\scriptsize 24};}
{    \filldraw[fill=white] (ten) circle (2pt) node[below] {\scriptsize 25};}
{    \filldraw[fill=red] (eleven) circle (2pt) node[above] {\scriptsize 26};}
{    \filldraw[fill=white] (twelve) circle (2pt) node[left] {\scriptsize 27};}
{    \filldraw[fill=white] (thirteen) circle (2pt) node[below] {\scriptsize 28};}
{    \filldraw[fill=white] (fourteen) circle (2pt) node[below] {\scriptsize 29};}
{    \filldraw[fill=white] (fifteen) circle (2pt) node[below] {\scriptsize 30};}
{    \filldraw[fill=white] (sixteen) circle (2pt) node[right] {\scriptsize 31};}
\end{tikzpicture}
\end{center}
\end{proof}

\subsection{$d=3,n=13$}\label{pack313}

\begin{claim}\label{claim:packing313}
The following inversive coordinates generate a subgroup of the group of isometries obtained by Vinberg's algorithm in \cite{Mcl10}.

\setcounter{MaxMatrixCols}{15}
\begin{align}
\label{coords13}\begin{array}{c|cccccccccccccc|l}
&\hat{b} & b & bx_1 & bx_2 & bx_3 & bx_4 & bx_5 & bx_6 & bx_7 & bx_8 & bx_9 & bx_{10} & bx_{11} & bx_{12} & \text{def'd as}:\\\hline
23& 0& 0& 0& 0& 0& 0& 0& 0& 0& 0& 0& 0& 0& -1 & 2.3.2.13\\
24& 0& 0& 0& 0& 0& 0& 0& 0& 0& 0& 0& 0& -\alpha & \alpha & 2.3.2.12\\
25& 0& 0& 0& 0& 0& 0& 0& 0& 0& 0& 0& -\alpha & \alpha & 0 & 2.3.2.11\\
26& 0& 0& 0& 0& 0& 0& 0& 0& 0& 0& -\alpha & \alpha & 0& 0 & 2.3.2.10\\
27& 0& 0& 0& 0& 0& 0& 0& 0& 0& -\alpha & \alpha & 0& 0& 0 & 2.3.2.9\\
28& 0& 0& 0& 0& 0& 0& 0& 0& -\alpha & \alpha & 0& 0& 0& 0 & 2.3.2.8\\
29& 0& 0& 0& 0& 0& 0& 0& -\alpha & \alpha & 0& 0& 0& 0& 0 & 2.3.2.7\\
30& 0& 0& 0& 0& 0& 0& -\alpha & \alpha & 0& 0& 0& 0& 0& 0 & 2.3.2.6\\
31& 0& 0& 0& 0& 0& -\alpha & \alpha & 0& 0& 0& 0& 0& 0& 0 & 2.3.2.5\\
32& 0& 0& -\alpha & 0& 0& \alpha & 0& 0& 0& 0& 0& 0& 0& 0 & 2.3.2.4\\
33& -\alpha & \alpha & \alpha & 0& 0& 0& 0& 0& 0& 0& 0& 0& 0& 0 & 1\\
34& \sqrt{2} & \sqrt{2} & 0& \gamma & \gamma & 0& 0& 0& 0& 0& 0& 0& 0& 0 & \beta\\
35& \hat{b}_{35} & b_{35} & 0& 0& 0& 0& 0& 0& 0& 0& 0& 0& 0& 0 & 14\\
36& \hat{b}_{35} & -b_{35} & 0& \sqrt{2} & 0& 0& 0& 0& 0& 0& 0& 0& 0& 0 & 2.14\\
37& \hat{b}_{35} & b_{35} & \alpha & \alpha & \alpha & \alpha & 0& 0& 0& 0& 0& 0& 0& 0 & 2.3.2.15\\
38& \hat{b}_{38} & b_{38} & \gamma & \sqrt{6} & \sqrt{6} & \gamma & \gamma & \gamma & \gamma & \gamma & \gamma & \gamma & \gamma & \gamma & 2.3.2.21\\
39& \hat{b}_{39} & b_{39} & \gamma & \gamma & \gamma & \gamma & \gamma & \gamma & \gamma & \gamma & 0& 0& 0& 0 & 2.3.2.16\\
40& \hat{b}_{40} & b_{40} & \sqrt{6} & \sqrt{6} & \sqrt{6} & \sqrt{6} & \sqrt{6} & \sqrt{6} & \gamma & \gamma & \gamma & \gamma & \gamma & \gamma & 2.3.2.22\\
41& \hat{b}_{41} & b_{41} & 1 & 1 & 1 & 1 & 1 & 1 & 1 & 1 & 1 & 1 & 1 & 1 & 2.3.2.20\\
42& \hat{b}_{42} & b_{42} & \frac{3\sqrt{2}}{2} & \frac{3\sqrt{2}}{2} & \frac{3\sqrt{2}}{2} & \frac{3\sqrt{2}}{2} & \frac{3\sqrt{2}}{2} & \frac{3\sqrt{2}}{2} & \frac{3\sqrt{2}}{2} & \alpha & \alpha & \alpha & \alpha & \alpha & 2.3.2.19\\
43& \hat{b}_{43} & b_{43} & \alpha & \alpha & \frac{3\sqrt{2}}{2} & \alpha & \alpha & \alpha & \alpha & \alpha & \alpha & \alpha & 0 & 0 & 2.3.2.17\\
44& \hat{b}_{44} & b_{44} & \alpha & \alpha & \frac{3\sqrt{2}}{2} & \alpha & \alpha & \alpha & \alpha & \alpha & \alpha & \alpha & \alpha & \alpha & 2.3.2.18
\end{array}
\end{align}
\end{claim}
\footnotetext{$(2.1.2.14).3.(2.1).\nye{14}$}

for $\alpha=\frac{\sqrt{2}}{2}$, $\beta$ denoting $(2.1.2.14).3.(2.1).\nye{14}$, $\gamma=\frac{\sqrt{6}}{2}$, and the following values: 

\begin{align}
\begin{array}{c|cc}
k & \hat{b}_k & b_k\\\hline
35& \frac{\sqrt{2}+\sqrt{6}}{2} & \frac{\sqrt{2}-\sqrt{6}}{2}\\
38& 4\sqrt{2}+\sqrt{6} & 4\sqrt{2}-\sqrt{6}\\
39& \frac{5\sqrt{2}+\sqrt{6}}{2} & \frac{5\sqrt{2}-\sqrt{6}}{2}\\
40& 5\sqrt{2}+\sqrt{6} & 5\sqrt{2}-\sqrt{6}\\
41& 2\sqrt{3}+1 & 2\sqrt{3}-1\\
42& \frac{5\sqrt{6}+3\sqrt{2}}{2} & \frac{5\sqrt{6}-3\sqrt{2}}{2}\\
43& \sqrt{6}+\sqrt{2} & \sqrt{6}-\sqrt{2}\\
44& \frac{3\left(\sqrt{6}+\sqrt{2}\right)}{2} & \frac{3\left(\sqrt{6}-\sqrt{2}\right)}{2} 
\end{array}
\end{align}

This is \eqref{coords13}'s Gram matrix: 

\setcounter{MaxMatrixCols}{25}
\begin{align}
\left(
\begin{array}{cccccccccccccccccccccc}
 \sharp & \alpha & 0 & 0 & 0 & 0 & 0 & 0 & 0 & 0 & 0 & 0 & 0 & 0 & 0 & \gamma
   & 0 & \gamma & 1 & \alpha & 0 & \alpha \\
 \alpha & \sharp & \frac{1}{2} & 0 & 0 & 0 & 0 & 0 & 0 & 0 & 0 & 0 & 0 & 0 & 0 & 0 & 0 & 0
   & 0 & 0 & 0 & 0 \\
 0 & \frac{1}{2} & \sharp & \frac{1}{2} & 0 & 0 & 0 & 0 & 0 & 0 & 0 & 0 & 0 & 0 & 0 & 0 & 0 & 0 & 0 &
   0 & \frac{1}{2} & 0 \\
 0 & 0 & \frac{1}{2} & \sharp & \frac{1}{2} & 0 & 0 & 0 & 0 & 0 & 0 & 0 & 0 & 0 & 0 & 0 & 0 & 0 & 0 &
   0 & 0 & 0 \\
 0 & 0 & 0 & \frac{1}{2} & \sharp & \frac{1}{2} & 0 & 0 & 0 & 0 & 0 & 0 & 0 & 0 & 0 & 0 &
   \beta & 0 & 0 & 0 & 0 & 0 \\
 0 & 0 & 0 & 0 & \frac{1}{2} & \sharp & \frac{1}{2} & 0 & 0 & 0 & 0 & 0 & 0 & 0 & 0 & 0 & 0 & 0 & 0 &
   1 & 0 & 0 \\
 0 & 0 & 0 & 0 & 0 & \frac{1}{2} & \sharp & \frac{1}{2} & 0 & 0 & 0 & 0 & 0 & 0 & 0 & 0 & 0 &
   \beta & 0 & 0 & 0 & 0 \\
 0 & 0 & 0 & 0 & 0 & 0 & \frac{1}{2} & \sharp & \frac{1}{2} & 0 & 0 & 0 & 0 & 0 & 0 & 0 & 0 & 0 & 0 &
   0 & 0 & 0 \\
 0 & 0 & 0 & 0 & 0 & 0 & 0 & \frac{1}{2} & \sharp & \frac{1}{2} & 0 & 0 & 0 & 0 & \frac{1}{2} & 0 & 0
   & 0 & 0 & 0 & 0 & 0 \\
 0 & 0 & 0 & 0 & 0 & 0 & 0 & 0 & \frac{1}{2} & \sharp & \frac{1}{2} & 0 & 0 & 0 & 0 & 0 & 0 & 0 & 0 &
   0 & 0 & 0 \\
 0 & 0 & 0 & 0 & 0 & 0 & 0 & 0 & 0 & \frac{1}{2} & \sharp & 0 & \beta & \frac{1}{2} & 0 &
   \beta & 0 & 0 & 0 & 0 & \frac{1}{2} & 1 \\
 0 & 0 & 0 & 0 & 0 & 0 & 0 & 0 & 0 & 0 & 0 & \sharp & 1 & 0 & 0 & 2 & 2 & 4 & \sqrt{6} & 2 \sqrt{3} &
   \sqrt{3} & \sqrt{3} \\
 0 & 0 & 0 & 0 & 0 & 0 & 0 & 0 & 0 & 0 & \beta & 1 & \sharp & 0 & 0 & 1 & 1 & 2 &
   \gamma & \sqrt{3} & 0 & 0 \\
 0 & 0 & 0 & 0 & 0 & 0 & 0 & 0 & 0 & 0 & \frac{1}{2} & 0 & 0 & \sharp & 0 & \sqrt{3} & \sqrt{3} & 2
   \sqrt{3} & \frac{3\sqrt{2}}{2} & 3 & 1 & 2 \\
 0 & 0 & 0 & 0 & 0 & 0 & 0 & 0 & \frac{1}{2} & 0 & 0 & 0 & 0 & 0 & \sharp & 0 & 0 & 0 &
   \alpha & 0 & 0 & 0 \\
 \gamma & 0 & 0 & 0 & 0 & 0 & 0 & 0 & 0 & 0 & \beta & 2 & 1 & \sqrt{3} &
   0 & \sharp & 2 & 1 & 0 & \sqrt{3} & 0 & 0 \\
 0 & 0 & 0 & 0 & \beta & 0 & 0 & 0 & 0 & 0 & 0 & 2 & 1 & \sqrt{3} & 0 & 2 & \sharp & 1 &
   \gamma & 0 & 0 & \sqrt{3} \\
 \gamma & 0 & 0 & 0 & 0 & 0 & \beta & 0 & 0 & 0 & 0 & 4 & 2 & 2 \sqrt{3}
   & 0 & 1 & 1 & \sharp & 0 & 0 & 0 & \sqrt{3} \\
 1 & 0 & 0 & 0 & 0 & 0 & 0 & 0 & 0 & 0 & 0 & \sqrt{6} & \gamma & \frac{3\sqrt{2}}{2} &
   \alpha & 0 & \gamma & 0 & \sharp & \alpha & 0 &
   \alpha \\
 \alpha & 0 & 0 & 0 & 0 & 1 & 0 & 0 & 0 & 0 & 0 & 2 \sqrt{3} & \sqrt{3} & 3 & 0 &
   \sqrt{3} & 0 & 0 & \alpha & \sharp & 0 & 2 \\
 0 & 0 & \frac{1}{2} & 0 & 0 & 0 & 0 & 0 & 0 & 0 & \frac{1}{2} & \sqrt{3} & 0 & 1 & 0 & 0 & 0 & 0
   & 0 & 0 & \sharp & 0 \\
 \alpha & 0 & 0 & 0 & 0 & 0 & 0 & 0 & 0 & 0 & 1 & \sqrt{3} & 0 & 2 & 0 & 0 & \sqrt{3}
   & \sqrt{3} & \alpha & 2 & 0 & \sharp \\
\end{array}
\right)
\end{align}

for $\sharp=-1,\alpha=\frac{\sqrt{2}}{2},\beta=\frac{\sqrt{3}}{2},\gamma=\frac{\sqrt{6}}{2}$. (The use of variable names is purely due to formatting constraints due to the size of the Gram matrix.)

\begin{lemma}
\eqref{coords13} has empty interior in $\R^{12}$, and extends by Poincar\'{e} extension to a hyperbolic polytope of finite volume. 
\end{lemma}
\begin{proof}
We first show that the configuration has empty interior: 
\begin{align}
-x_{12}&>0&&\tag{3.13.23}\\
\implies x_{12}&<0\nonumber\\
\label{ineq31320}\left(\frac{2\sqrt{3}+1}{7}\right)^2-\sum\limits_{i=1}^{12}\left(\frac{2\sqrt{3}+1}{7}-x_i\right)^2&>0&&\tag{3.11.42}\\
\left(\frac{2\sqrt{3}+1}{7}\right)^2>\sum\limits_{i=1}^{12}\left(\frac{2\sqrt{3}+1}{7}-x_i\right)^2&\ge\left(\frac{2\sqrt{3}+1}{7}-x_{12}\right)^2,\nonumber\\
&>\left(\frac{2\sqrt{3}+1}{7}\right)^2,\nonumber
\end{align}

\noindent a contradiction. Hence no point $(x_i)_{i=1}^{12}\in\R^{12}$ lies in the mutual interior of the specified configuration. 

\eqref{ineq31320} gives bounds for each coordinate:
\begin{align*}
&\left(\frac{2\sqrt{3}+1}{7}\right)^2-\sum\limits_{i=1}^{12}\left(\frac{2\sqrt{3}+1}{7}-x_i\right)^2>0\\
\implies&\left(\frac{2\sqrt{3}+1}{7}\right)^2>\sum\limits_{i=1}^{12}\left(\frac{2\sqrt{3}+1}{7}-x_i\right)^2\ge\left(\frac{2\sqrt{3}+1}{7}-x_i\right)^2\\
\implies&0<x_i<\frac{4\sqrt{3}+2}{7}
\end{align*}

\noindent for $1\le i\le12$. Since the intersection of the respective Poincar\'{e} extensions of the circles is bounded and does not meet the boundary of $\H^{13}$, it must be of finite volume.
\end{proof}

\begin{thm}
\eqref{coords13} generates a sphere packing in $\R^{12}$ through the cluster $\{35\}$. 
\end{thm}
\begin{proof}
Application of \thmref{structthm} to the following Coxeter diagram proves the result. 
\begin{center}
\begin{tikzpicture}[scale=1.5]%n=13
    \coordinate (2) at (0,0);
    \coordinate (3) at (1,0);
    \coordinate (15) at (2,0);
    \coordinate (18) at (3.5,0);
    \coordinate (4) at (1,-1);
    \coordinate (5) at (2.5,-1);
    \coordinate (6) at (4.5,-1);
    \coordinate (1) at (0,1);
    \coordinate (22) at (1,1);
    \coordinate (12) at (3,1);
    \coordinate (13) at (3.75,1);
    \coordinate (14) at (5,1);
    \coordinate (21) at (0,3);
    \coordinate (23) at (1,2);
    \coordinate (11) at (3,2);
    \coordinate (17) at (5,2);
    \coordinate (20) at (1,3);
    \coordinate (16) at (2,3);
    \coordinate (9) at (3,3);
    \coordinate (19) at (5,3);
    \coordinate (8) at (3,4);
    \coordinate (7) at (4,4);
    
    \draw[thick, double distance=2pt] (1) -- (2);
    \draw[dashed] (1) -- (17);
    \draw[dashed] (1) -- (19);
    \draw[ultra thick] (1) -- (20);
    \draw[thick, double distance=2pt] (1) -- (21);
    \draw[thick, double distance=2pt] (1) -- (23);
    \draw[thick] (2) -- (3);
    \draw[thick] (3) -- (4);
    \draw[thick] (3) -- (22);
    \draw[thick] (4) -- (5);
    \draw[thick] (5) -- (6);
    \draw[thick] (5) -- (18);
    \draw[thick] (6) -- (7);
    \draw[ultra thick] (6) -- (23);
    \draw[thick] (7) -- (8);
    \draw[thick,double distance=2.5 pt] (7) -- (19);
    \draw[thick,double distance=0.3 pt] (7) -- (19);
    \draw[thick] (8) -- (9);
    \draw[thick] (9) -- (11);
    \draw[thick] (9) -- (16);
    \draw[thick] (11) -- (12);
    \draw[thick] (12) -- (14);
    \draw[thick] (12) -- (15);
    \draw[thick,double distance=2.5pt] (12) -- (17);
    \draw[thick,double distance=0.3pt] (12) -- (17);
    \draw[thick] (12) -- (22);
    \draw[ultra thick] (13) -- (14);
    \draw[dashed] (13) -- (18);
    \draw[dashed] (13) -- (17);
    \draw[dashed] (13) -- (19);
    \draw[dashed] (13) -- (20);
    \draw[ultra thick] (14) -- (17);
    \draw[ultra thick] (14) -- (18);
    \draw[dashed] (14) -- (19);
    \draw[dashed] (14) -- (20);
    \draw[dashed] (15) -- (17);
    \draw[dashed] (15) -- (18);
    \draw[dashed] (15) -- (19);
    \draw[dashed] (15) -- (20);
    \draw[ultra thick] (15) -- (22);
    \draw[dashed] (15) -- (23);
    \draw[thick, double distance=2pt] (16) -- (20);
    \draw[ultra thick] (17) -- (19);
    \draw[dashed] (17) -- (18);
    \draw[dashed] (17) -- (21);
    \draw[ultra thick] (18) -- (19);
    \draw[dashed] (18) -- (20);
    \draw[dashed] (18) -- (23);
    \draw[dashed] (19) -- (23);
    \draw[thick,double distance=2pt] (20) -- (23);
    \draw[thick,double distance=2pt] (20) -- (21);
    \draw[dashed] (21) -- (23);
    
{    \filldraw[fill=white] (1) circle (2pt) node[left] {\scriptsize 23};}
{    \filldraw[fill=white] (2) circle (2pt) node[below] {\scriptsize 24};}
{    \filldraw[fill=white] (3) circle (2pt) node[right] {\scriptsize 25};}
{    \filldraw[fill=white] (4) circle (2pt) node[below] {\scriptsize 26};}
{    \filldraw[fill=white] (5) circle (2pt) node[below] {\scriptsize 27};}
{    \filldraw[fill=white] (6) circle (2pt) node[below] {\scriptsize 28};}
{    \filldraw[fill=white] (7) circle (2pt) node[above] {\scriptsize 29};}
{    \filldraw[fill=white] (8) circle (2pt) node[above] {\scriptsize 30};}
{    \filldraw[fill=white] (9) circle (2pt) node[right] {\scriptsize 31};}
{    \filldraw[fill=white] (11) circle (2pt) node[right] {\scriptsize 32};}
{    \filldraw[fill=white] (12) circle (2pt) node[below] {\scriptsize 33};}
{    \filldraw[fill=white] (13) circle (2pt) node[below] {\scriptsize 34};}
{    \filldraw[fill=red] (14) circle (2pt) node[below] {\scriptsize 35};}
{    \filldraw[fill=white] (15) circle (2pt) node[below] {\scriptsize 36};}
{    \filldraw[fill=white] (16) circle (2pt) node[above] {\scriptsize 37};}
{    \filldraw[fill=white] (17) circle (2pt) node[right] {\scriptsize 38};}
{    \filldraw[fill=white] (18) circle (2pt) node[right] {\scriptsize 39};}
{    \filldraw[fill=white] (19) circle (2pt) node[right] {\scriptsize 40};}
{    \filldraw[fill=white] (20) circle (2pt) node[above] {\scriptsize 41};}
{    \filldraw[fill=white] (21) circle (2pt) node[left] {\scriptsize 42};}
{    \filldraw[fill=white] (22) circle (2pt) node[left] {\scriptsize 43};}
{    \filldraw[fill=white] (23) circle (2pt) node[right] {\scriptsize 44};}
    
\end{tikzpicture}
\end{center}
\end{proof}

\section{Converting into inversive coordinates}

Often, authors use alternate coordinate systems when working with $\H^{n+1}$ and Vinberg's algorithm. In this section, we specify the transformations used for a given quadratic form, in order to preserve the properties of each space; namely, if it is known that for some quadratic form $A$, all vectors $v\in V$ have $\<v,v\>_A$ with some property, then we wish to find $f_A$ such that $\<f_A(v),f_A(v)\>_Q$ satisfies an analogue to that property. 

\subsection{Conversion of \cite{Mcl13}'s $\hat{Bi}$ coordinates}
\label{mcl bi conversion}

This is relevant to \secref{bianchi packings}. The vectors produced by \cite{Mcl13} were obtained using the quadratic form 
\begin{equation}\label{Mcl quad}
    f=
    \begin{cases}
        -2x_1x_2+2x_3^2+2mx_4^2 & \text{if } m \equiv 1,2 \text{ (mod 4)} \\
        -2x_1x_2+2x_3^2+2x_3x_4+\frac{m+1}{2}x_4^2 & \text{if } m \equiv 3 \text{ (mod 4)}
    \end{cases}
\end{equation}
for each $\hat{Bi}(m)$. Our first step in obtaining extended Bianchi group packings was to convert these coordinates to coordinates that correspond to our quadratic form $Q=-1$, by which we mean that all vectors $v$ satisfy $\ip{v}{v}_Q=-1$ (see \defref{Q} and \cite{Kon17}). To recap, this quadratic form arose directly from \defref{inversion} of sphere inversion, which led to the equation
\begin{equation}\label{Q derivation}
    \hat{b}b-|bz|^2=-1.
\end{equation}
In order to generate circle packings from $\hat{Bi}(m)$, we converted  \cite{Mcl13}'s coordinates (after normalizing their lengths) to fit the 2-dimensional version of (\ref{Q derivation}), $\hat{b}b-(b\bar{x})^2-(b\bar{y})^2$, in the following manner:

\begin{equation}
    \begin{cases}
    (x_1,x_2,x_3,x_4 \sqrt{m}) \mapsto (\hat{b},b,b\bar{x},b\bar{y}) & \text{if } m \equiv 1,2 \text{ (mod 4)} \\
    (x_1,x_2,x_3 + \frac{x_4}{2},\frac{x_4 \sqrt{m}}{2}) \mapsto (\hat{b},b,b\bar{x},b\bar{y}) & \text{if } m \equiv 3 \text{ (mod 4)}
    \end{cases}
\end{equation}

\subsection{Conversion of \cite{Vin72,Mcl10}'s coordinates}\label{d qf}

This is relevant to \secref{hd}. In that context, \cite{Vin72,Mcl10} use quadratic forms $-dx_0^2+\sum\limits_{i=1}^nx_i^2$ and vectors $x=(x_i)_{i=0}^n\in\R^{n+1}$ for which $\<x,x\>\in\N$. We apply the following conversion: 
\begin{align}
\label{vinconv}f(x)=(\hat{x}_0\sqrt{d}+\hat{x}_1,\hat{x}_0\sqrt{d}-\hat{x}_1,\hat{x}_2,\dots,\hat{x}_n)
\end{align}

\noindent where $\hat{x}=x/\sqrt{\<x,x\>}$ with components $\hat{x}_0,\dots,\hat{x}_n$. 

\begin{lemma}
\eqref{vinconv} corresponds to valid inversive coordinates. 
\end{lemma}
\begin{proof}
\begin{align*}
f(x)Qf(x)^T&=d\hat{x}_0^2-\hat{x}_1^2-\sum\limits_{i=2}^nx_i^2\\
&=d\left(\frac{x_0}{\sqrt{\<x,x\>}}\right)^2-\sum\limits_{i=1}^n\left(\frac{x_i}{\sqrt{\<x,x\>}}\right)^2\\
&=\frac{dx_0^2-\sum\limits_{i=1}^nx_i^2}{\<x,x\>}\\
&=-\frac{\<x,x\>}{\<x,x\>}\\
&=-1.
\end{align*}
\end{proof}

Therefore, for given $d$, \eqref{vinconv} gives the function used to convert to inversive coordinates, preserving the properties of the domain inner product space. 

\section{A note on implementing the Lobachevsky function}

As in e.g. \cite{Mil82,Vin93}, we have: 

\begin{Def}[Lobachevsky function]\label{lob}
The \emph{Lobachevsky function} is the integral
\begin{align}
\label{lob1} L(\theta)=\int\limits_0^\theta\log\abs{2\sin u}du.
\end{align}
\end{Def}

\cite{Mil82} discusses the importance of this function in computing exact hyperbolic volume, specifically in the case of ideal tetrahedra in $\H^3$, and \cite{Vin93} provides further examples of some general computations for other hyperbolic solids. Per \cite{Mil82}, the following are also true of small $\theta$: 

\begin{align}
\label{lob2}L(\theta)&=\theta\left(1-\log\abs{2\theta}+\sum\limits_{n\ge1}\frac{B_n(2\theta)^{2n}}{2n(n+1)!}\right)\\
\label{lob3}L(\theta)&=\half\sum\limits_{n\ge1}\frac{\sin(2n\theta)}{n^2}
\end{align}

\noindent with \eqref{lob2} especially recommended for use in computation. However, comparing the runtimes of these functions using Mathematica implementations reveals that not only does the error in \eqref{lob2} become non-negligible for practically-sized $\theta$, but also that in the Mathematica architecture, \eqref{lob3} vastly outperforms \eqref{lob1} and \eqref{lob2} on $\theta\in[0,2\pi)$, and that Mathematica optimizes the infinite sum to run faster than a sum with a built-in cutoff; i.e., 

\begin{align}
\label{lob4}L(\theta,N)=\half\sum\limits_{n=1}^N\frac{\sin(2n\theta)}{n^2}
\end{align}

\noindent evaluates slower than \eqref{lob3} even for $N$ as small as 1000. 

\newpage

\section*{Acknowledgements}

We would further like to thank Lazaros Gallos, Parker Hund, and the rest of the DIMACS REU organization, and the Rutgers Mathematics Department for their support, without which this work would not have been possible. We would like to thank Kei Nakamura and Alice Mark for taking the time to discuss this material with us. We would most of all like to thank Professor Alex Kontorovich for his guidance and mentorship on this work.


\begin{thebibliography}{9}
\bibitem{All66}
N. D. Allan. ``The problem of the maximality of arithmetic groups." In: \emph{Algebraic Groups and Discontinuous Subgroups (Proc. Sympos. Pure Math., Boulder, Colo., 1965)} (1966), pp. 104--109.

\bibitem{Bia92}
Luigi Bianchi. ``Sui gruppi di sostituzioni lineari con coe cienti appartenenti a corpi quadratici immaginari." In: \emph{Math. Ann.} 40.3 (1892), pp. 332--412. 

\bibitem{BM13}
Mikhail Belolipetsky and John Mcleod. ``Reflective and quasi-reflective {B}ianchi groups." In \emph{Transform. Groups} 18.4 (2013), pp. 971--994. 

\bibitem{BS04}
Alexander I. Bobenko and Boris A. Springborn. ``Variational principles for circle patterns and Koebe's theorem." In: \emph{Trans. Amer. Math. Soc} 356 (2004), pp. 659--689. 

\bibitem{Ver91}
Yves de Colin de Verdi\`{e}re. ``Un principe variationnel pour les empilements de cercles." In: \emph{Inventiones mathematicae} 104.3 (1991), pp. 655--669. {\sc URL}: \url{http://eudml.org/doc/143902}. 

\bibitem{KN18}
Alex Kontorovich and Kei Nakamura. ``Geometry and arithmetic of crystallographic sphere packings." In: \emph{Proceedings of the National Academy of Sciences} (2018). {\sc ISSN}: 00278424. {\sc URL}: \url{https://www.pnas.org/content/early/2018/12/21/1721104116}. 

\bibitem{Kon17}
Alex Kontorovich. \emph{Letter to Bill Duke}. 2017. {\sc URL}: \url{https://math.rutgers.edu/~alexk/files/LetterToDuke.pdf}. 

\bibitem{Mcl10}
John Mcleod. ``Hyperbolic reflection groups associated to the quadratic forms $-3x_0^2+x_1^2+\dots+x_n^2$." In: \emph{Geom. Dedicata} 152 (2010), pp. 1--16. 

\bibitem{Mcl13}
John Mcleod. ``Arithmetic Hyperbolic Reflection Groups." PhD thesis. Durham University, 2013. 

\bibitem{Mil82}
John Milnor, ``Hyperbolic Geometry: the First 150 Years." In: \emph{Bulletin of the American Mathematical Society} 6.1 (Jan. 1982), pp. 9--23. 

\bibitem{Riv86}
Igor Rivin. ``On Geometry of Convex Polyhedra in Hyperbolic 3-Space." PhD thesis. Princeton University, 1986. {\sc URL}: \url{http://gateway.proquest.com/openurl?url_ver=Z39.88-2004&rft_val_fmt=info:ofi/fmt:kev:mtx:dissertation&res_dat=xri:pqdiss&rft_dat=xri:pqdiss:8626178}. 

\bibitem{Riv94}
Igor Rivin. ``Euclidean Structures on Simplicial Surfaces and Hyperbolic Volume." In: \emph{Annals of Mathematics} 139.3 (1994), pp. 553--580. {\sc ISSN}: 0003486X, {\sc URL}: \url{http://www.jstor.org/stable/2118572}. 

\bibitem{Ruz90}
O. P. Ruzmanov, ``Subgroups of reflections in Bianchi groups." In: \emph{Uspekhi Mat. Nauk} 45 (1990), pp. 189-190. 

\bibitem{Sha90}
M. K. Shaiheev. ``Reflective subgroups in Bianchi groups." In: \emph{Selecta Math. Soviet} 9 (1990), pp. 315--322. 

\bibitem{Vin67}
\`{E}rnest Vinberg. ``Discrete Groups Generated by Reflections in Loba\^{c}evski\u{i} Spaces." In: \emph{Mathematics of the USSR-Sbornik} 1.3 (1967). {\sc URL}: \url{http://stacks.iop.org/0025-5734/1/i=3/a=A08}. 

\bibitem{Vin72}
\`{E}rnest Vinberg. ``The groups of units of certain quadratic forms." In: \emph{Mat. Sb. (N.S.)} 87.129 (1972), pp.18--36. 

\bibitem{Vin90}
\`{E}rnest Vinberg. ``Reflective subgroups in Bianchi groups." In: \emph{Selecta Math. Soviet} 9 (1990), pp. 309--314. 

\bibitem{Vin93}
\`{E}rnest Vinberg, ``Volumes of non-Euclidean polyhedra." In: \emph{Russ. Math. Surv.} 2.48 (1993), pp. 15--45. 

\bibitem{Zie04}
G{\"u}nter M. Ziegler. ``Convex Polytopes: Extremal Constructions and $f$-Vector Shapes." In: \emph{IAS/Park City Mathematics Series} 14 (2004). 
\end{thebibliography}
\end{document}